\documentclass[11pt, dina4]{article}

\usepackage{amssymb, amsmath,amsthm}
\usepackage[utf8]{inputenc}
\usepackage{lineno}
\usepackage{graphicx, epsfig}
\DeclareGraphicsExtensions{.eps}
\usepackage{calrsfs}
\usepackage{multirow}
\usepackage{hyperref}
\usepackage{cite}
\usepackage{comment}
\usepackage{upgreek}
\usepackage{epstopdf}
\usepackage{float}
\usepackage{diagbox}
\usepackage{lscape}
\usepackage{subfigure}
\usepackage{dsfont}
\usepackage{color,xcolor}
\usepackage{soul}
\usepackage[margin=1in]{geometry}
\usepackage{tikz}
\usepackage{amsbsy} 
\usepackage{dutchcal}

\newcommand{\bu}{{\bf u} }
\newcommand{\bv}{{\bf v} }

\newcommand{\md}{\,\mathrm{d}}

\newcommand{\bj}{{\bf j}}
\newcommand{\bxi}{\boldsymbol{\xi}}

\newcommand{\mathr}{\mathbb{R}}
\newcommand{\mathz}{\mathbb{Z}}
\newcommand{\mathf}{\mathcal{F}_{h}}

\newcommand{\og}{\omega}
\newcommand{\Og}{\Omega}
\newcommand{\fl}[2]{\frac{#1}{#2}}

\newcommand{\rmone}{\uppercase\expandafter{\romannumeral1}}
\newcommand{\rmtwo}{\uppercase\expandafter{\romannumeral2}}
\newcommand{\rmthree}{\uppercase\expandafter{\romannumeral3}}
\newcommand{\rmfour}{\uppercase\expandafter{\romannumeral4}}
\newcommand{\rmfive}{\uppercase\expandafter{\romannumeral5}}
\newcommand{\dt}{\delta}
\newcommand{\gm}{\gamma}
\newcommand{\nn}{\nonumber}
\newcommand{\ap}{\alpha}

\newcommand{\veps}{\varepsilon}

\newcommand{\Dt}{\Delta}

\newcommand{\be}{\begin{equation}}
\newcommand{\ee}{\end{equation}}
\newcommand{\ba}{\begin{array}}
\newcommand{\ea}{\end{array}}
\newcommand{\bea}{\begin{eqnarray}}
\newcommand{\eea}{\end{eqnarray}}
\newcommand{\beas}{\begin{eqnarray*}}
\newcommand{\eeas}{\end{eqnarray*}}
\newtheorem{remark}{Remark}[section]
\newtheorem{definition}{Definition}[section]
\newtheorem{lemma}{Lemma}[section]
\newtheorem{theorem}{Theorem}[section]

\newtheorem{example}{Example}[section]
\newcommand{\bx}{{\bf x} }
\newcommand{\by}{{\bf y} }

\newcommand{\bb}{\vskip 10pt}
\definecolor{ForestGreen}{rgb}{0.0, 0.5, 0.0}

\definecolor{RoyalBlue}{cmyk}{0.35, 0.75, 0.5, 0.0}

\title{A novel and simple spectral method for nonlocal PDEs with the fractional Laplacian}

\author{Shiping Zhou\thanks{Department of Mathematics and Statistics, Missouri University of Science and Technology, Rolla, MO 65409 (Email:  szb5g@mst.edu)}, \ \
Yanzhi Zhang\thanks{Department of Mathematics and Statistics, Missouri University of Science and Technology, Rolla, MO 65409 (Email:  zhangyanz@mst.edu)}}

\begin{document}
\date{}
\maketitle

\begin{abstract}
We propose a novel and simple spectral method based on the semi-discrete Fourier transforms 
to discretize the fractional Laplacian $(-\Dt)^\fl{\ap}{2}$.  
Numerical analysis and experiments are provided to study its performance.  
Our method has the same symbol $|\boldsymbol\xi|^\alpha$ as the fractional Laplacian $(-\Delta)^\frac{\alpha}{2}$ at the discrete level, and thus it can be viewed as the exact discrete analogue of the fractional Laplacian. 
This {\it unique feature} distinguishes our method from other existing methods for the fractional Laplacian. 
Note that our method is different from the Fourier pseudospectral methods in the literature which are usually limited to periodic boundary conditions (see Remark \ref{remark0}). 
Numerical analysis shows that our method can achieve a spectral accuracy. 
The stability and convergence of our method in solving the fractional Poisson equations were analyzed. 
Our scheme yields a multilevel Toeplitz stiffness matrix,  and thus fast algorithms can be developed for efficient matrix-vector multiplications. 
The computational complexity is ${\mathcal O}(2N\log(2N))$, and the memory storage is ${\mathcal O}(N)$ with $N$ the total number of points. 
Extensive numerical experiments verify our analytical results and demonstrate the effectiveness of our method in solving various problems.  
\end{abstract}




\section{Introduction}
\label{section1}

The fractional Laplacian $(-\Dt)^\fl{\ap}{2}$, representing the infinitesimal generator of a symmetric $\alpha$-stable L\'evy process,  can be viewed as a nonlocal generalization of the classical Laplacian. 
It has found applications in many areas, including quantum mechanics, turbulence, plasma,  finance, and so on. 
Compared to its classical counterpart, the nonlocality of the fractional Laplacian introduces considerable challenges in numerical approximation and computer implementation. 
Over the past decades,  many numerical studies have been reported for the fractional Laplacian; see \cite{Gao2014, Huang2014, Duo2015, Duo2018, Duo2019-FDM, Minden2020, Hao2021, Bonito2019, Acosta2017,  Ainsworth2017,  Burkardt2021, Chen2018, Hao2021, Hao2021sharp, Sheng2020, Tang2018, Tang2020, Wu0021,  Mao2017, Wu2022, Xu0018, Kirkpatrick2016} and references therein. 
In this paper, we propose a novel and simple spectral method based on the semi-discrete Fourier transforms to discretize the fractional Laplacian $(-\Dt)^\fl{\ap}{2}$.  
Note that our method is essentially different from the Fourier pseudospectral methods in the literature \cite{Kirkpatrick2016, Duo2016}; see more discussion in Remark \ref{remark0}. 

The fractional Laplacian can be defined as a pseudo-differential operator with symbol $|{\boldsymbol \xi}|^\ap$ \cite{Landkof1972, Samko1993, Kwasnicki2017},  i.e.,
\begin{equation}\label{pseudo}
(-\Dt)^{\fl{\alpha}{2}} u(\bx) = \mathcal{F}^{-1}\big[ |{\boldsymbol \xi}|^\ap \mathcal{F} [u]\big], \quad \  \text{ for }\alpha>0,
\end{equation}
where $\mathcal{F}$ represents the Fourier transform over $\mathbb{R}^d$ with its associated inverse transform denoted as $\mathcal{F}^{-1}$.  
In the special case of $\ap = 2$,  the definition in (\ref{pseudo}) reduces to the spectral representation of the classical negative Laplacian $-\Dt$.  
In the literature \cite{Landkof1972, Samko1993, Kwasnicki2017}, the fractional Laplacian is also defined in the form of a hypersingular integral:
\begin{equation}\label{integralFL}
(-\Dt)^{\fl{\alpha}{2}} u(\bx) = c_{d,\alpha}\text{ P.V.} \int_{\mathbb{R}^d} \fl{u(\bx)-u(\by)}{|\bx-\by|^{d+\alpha}}\md \by, \quad \ \text{ for } \alpha \in(0,2),
\end{equation}
where P.V. stands for the Cauchy principal value, and the normalization constant $c_{d,\alpha}$ is given by $c_{d,\alpha} = {2^{\alpha-1} \ap\,\Gamma{\big(\fl{\alpha+d}{2}\big)}}/\big[{\sqrt{\pi^{d}}\,\Gamma{\big(1-\fl{\alpha}{2}\big)}}\big]$
with $\Gamma(\cdot)$ denoting the Gamma function. 
Note that the hypersingular integral definition in (\ref{integralFL}) holds only for  power $0 < \ap < 2$. 
For $\ap \in (0, 2)$, the two definitions (\ref{pseudo}) and (\ref{integralFL}) of the fractional Laplacian  are equivalent  in the Schwartz space $S({\mathbb R}^d)$  \cite{Samko1993, Kwasnicki2017}.  
More discussion on  the fractional Laplacian and its relation to other nonlocal operators can be found in \cite{Kwasnicki2017, Duo2019comparison} and references therein.  
In this work, we will focus on the pseudo-differential representation of the fractional Laplacian $(-\Dt)^\fl{\ap}{2}$ in (\ref{pseudo}). 

In the literature, numerical methods for the fractional Laplacian $(-\Dt)^\fl{\ap}{2}$ can be roughly classified into  three groups:  finite element methods,  finite difference methods, and spectral methods. 
Finite element/difference methods are usually developed based on the hypersingular integral definition in (\ref{integralFL}).  
For example, various finite element methods are proposed in \cite{Delia2013, Tian2016, Acosta2017,  Ainsworth2017, Ainsworth2017, Bonito2019} to solve nonlocal PDEs with the fractional Laplacian.  
In \cite{Huang2014,  Gao2014, Huang0016, Duo2018, Duo2019-FDM, Duo2019-TFL, Minden2020, Hao2021, Wu2022}, finite  difference methods are introduced based on different discretization strategies, and most of them have a second-order accuracy. 
On the other side, spectral methods have recently gained a lot of attention in solving nonlocal/fractional PDEs. 
They can achieve high accuracy with less number of points and thus may overcome the challenges caused by the nonlocality.  
In \cite{Kirkpatrick2016, Duo2016}, the Fourier pseudospectral methods are proposed for fractional PDEs with periodic boundary conditions.  
The Jacobi spectral methods are introduced in \cite{Xu0018, Hao2021sharp} for problems defined on a unit ball with extended homogeneous Dirichlet boundary conditions. 
Spectral methods based on Hermite functions are proposed in \cite{Mao2017, Tang2018} for fractional PDEs  on unbounded domains,  while the generalized Laguerre functions are used in \cite{Chen2018}.  
Later, the spectral method with rational basis is proposed in \cite{Tang2020}. 
It is pointed out in \cite{Sheng2020} that due to the singular and nonseparable factor $|\bxi|^{\ap}$ in the Fourier definition, these methods become extremely complicated for $d \ge 2$. 
Moreover,  a new class of meshfree spectral methods are proposed based on radial basis functions \cite{Burkardt2021, Wu0021, Wu2022}. 
They combine both definitions (\ref{pseudo}) and (\ref{integralFL}) of the fractional Laplacian and thus can solve problems with nonhomogeneous Dirichlet boundary conditions.

In this paper, we propose a new and simple spectral method to discretize the fractional Laplacian $(-\Dt)^\fl{\ap}{2}$ and apply it to study nonlocal elliptic problems. 
The key idea of our method is to apply the semi-discrete Fourier transforms to approximate the pseudo-differential representation of the fractional Laplacian in (\ref{pseudo}). 
\textbf{Our scheme can be viewed as a discrete pseudo-differential operator with symbol $|\boldsymbol\xi|^\ap$, and it provides an exact discrete analogue of the fractional Laplacian $(-\Dt)^\fl{\ap}{2}$. } 
This {\bf unique feature} distinguishes our method from other existing methods in the literature. 
The implementation of our method is simple and efficient.  
It yields a multilevel Toeplitz stiffness matrix,  and thus fast algorithms can be developed for  matrix-vector multiplications with computational cost of ${\mathcal O}(2N\log(2N))$ and memory cost of ${\mathcal O}(N)$, where $N$ is the total number of spatial points.  
Numerical analysis are provided to study the performance of our method. 
 We prove that our method has an accuracy of  ${\mathcal O}(h^{p+\gm-\ap+1/2})$ in approximating the fractional Laplacian, if $u \in C^{p, \gm}({\mathbb R})$ with $p\in{\mathbb N}^0$ and $\gm \in (0, 1]$.
Particularly, it has a spectral accuracy if $u \in C^\infty({\mathbb R^d})$. 
Moreover,  the stability and convergence analysis are provided for solving the fractional Poisson equations.  
Theses analytical results are verified and confirmed by our extensive numerical experiments. 
We also apply our method to solve fractional elliptic problems and study the coexistence of normal and anomalous diffusion problems. 
\begin{remark}\label{remark0}
We remark  that our method is essentially different from the Fourier pseudospectral methods in the literature \cite{Kirkpatrick2016, Duo2016}.
These existing methods are based on the {\bf discrete  (instead of semi-discrete)} Fourier transforms and   limited to periodic boundary conditions. 
\begin{figure}[ht!]
\centerline{
(a)\includegraphics[width=5.68cm, height = 3.96cm]{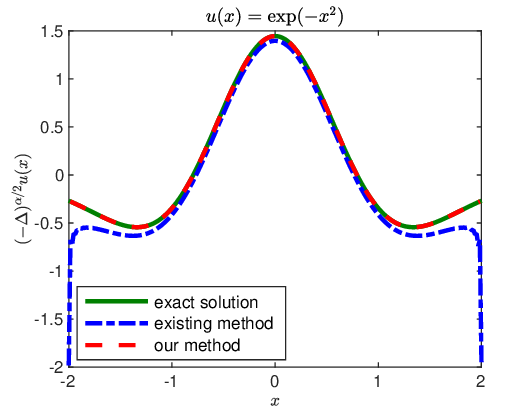}
\hspace{-5mm}
(b)\includegraphics[width=5.68cm, height = 3.96cm]{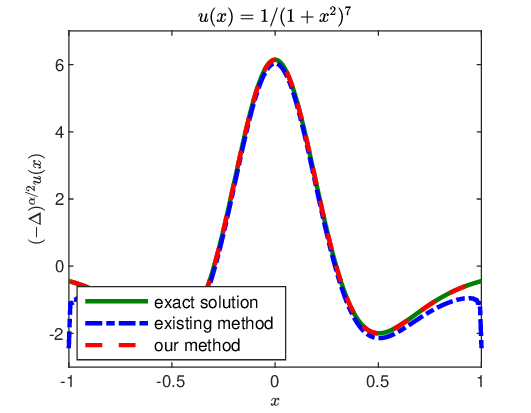}}
\centerline{
(c)\includegraphics[width=5.68cm, height = 3.96cm]{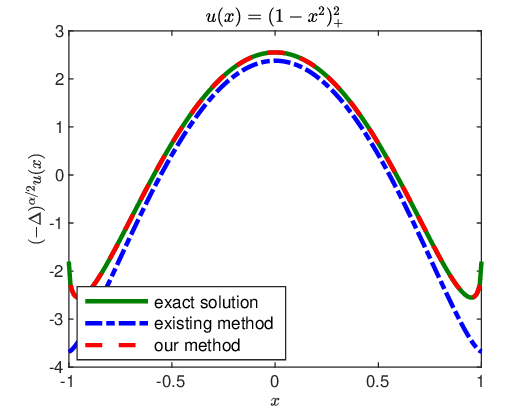}
\hspace{-5mm}
(d)\includegraphics[width=5.98cm, height = 3.96cm]{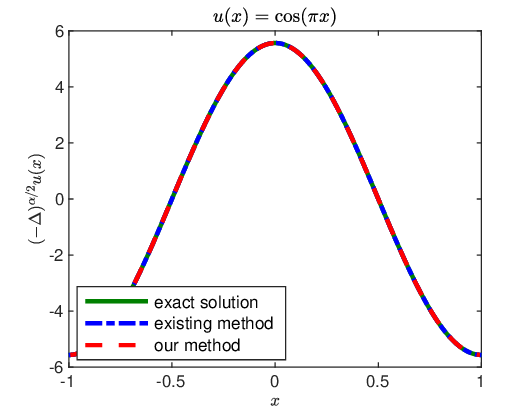}}
\caption{Comparison of our method and the existing Fourier pseudo-spectral methods in \cite{Kirkpatrick2016, Duo2016} for approximating $(-\Dt)^\fl{\ap}{2}u(x)$ on bounded domain. }
\label{Figure0}
\end{figure}
Figure \ref{Figure0} compares our method with these existing methods in approximating $(-\Dt)^\fl{\ap}{2}u$. 
It clearly shows that the existing methods in \cite{Kirkpatrick2016, Duo2016} require periodic boundary conditions to provide accurate approximation; see Figure \ref{Figure0} (d).  
\end{remark}

The paper is organized as follows. 
In Section \ref{section2}, we introduce our new spectral method for the one-dimensional Laplacian. 
Numerical analysis are presented in Section \ref{section3}. 
The generalizations of our method to high dimensions (i.e. $d > 1)$ are addressed in Section \ref{section4}. 
In Section \ref{section5}, we conduct numerical experiments to examine the performance of our method in approximating the classical and fractional Laplacians and in solving the fractional elliptic problems. 
Finally, some concluding remarks are made in Section \ref{section6}. 

\section{Spectral method}
\setcounter{equation}{0}
\label{section2}

Due to the pseudo-differential definition in (\ref{pseudo}), it is natural to introduce the Fourier transform-based methods to approximate the fractional Laplacian. 
For example,  the Fourier pseudospectral methods are introduced to solve the fractional Schr\"odinger equations  in \cite{Kirkpatrick2016, Duo2016, Duo2021} and the reaction-diffusion systems in \cite{Khudhair2022}.  
These methods can be {\it directly} implemented via the fast Fourier transforms, but they are limited to periodic bounded domains. 
In this section, we introduce a new and simple spectral method to discretize the fractional Laplacian $(-\Dt)^\fl{\ap}{2}$. 
Our method is developed based on the semi-discrete Fourier transforms.  
In contrast to those Fourier pseudospectral methods in \cite{Kirkpatrick2016, Duo2016}, our method is free of the constraint of periodic boundary conditions. 

To facilitate the discussion, we will first introduce our method for one-dimensional (i.e. $d = 1$) fractional Laplacian $(-\Dt)^\fl{\ap}{2}$. 
The generalization  to high dimensions (i.e. $d > 1$) will be presented in Section \ref{section4}. 
Let $h > 0$ denote the mesh size.  
In the one-dimensional cases, we define the grid points $x_j = jh$ for $j \in {\mathbb Z}$. 
Denote the one-dimensional grid function $\bu = \{u_{j}\}_{j\in\mathz}$,  which may or may not be an approximation to a continuous function. 
For grid functions $\bu$ and $\bv$, we define the discrete inner product and associated norm as
\begin{equation*}
   {\langle\bu,\,\bv\rangle_h} := h \sum_{j\,\in\,\mathz} u_{j}\,\overline{v}_{j}, \qquad \, \|\bu\|_{l^2} = \sqrt{\langle\bu,\,\bu\rangle_h},
\end{equation*}
where $\overline{v}_j$ represents the complex conjugate of $v_j$. 
Let  $l^2({\mathbb R}) = \big\{ \bv\,\big|\, \|\bv\|_{l^2} < \infty \big\}$.  

In the following, we will first introduce the definitions of semi-discrete Fourier transform and generalized hypergeometric functions, which play an important role in constructing our method. 

\begin{definition}[Semi-discrete Fourier transform]
\label{Def-SFT-1}
For a  grid function $\bu\in l^2({\mathbb R})$, the semi-discrete Fourier transform $\mathcal{F}_h$  is defined as
\begin{equation}\label{semiFT1D}
    \check{u}(\xi) = \big(\mathcal{F}_h [\bu]\big)(\xi) = h\sum_{{ j = -\infty}}^\infty u_{j}\,e^{-i \xi x_{j}}, \qquad \mbox{for} \ \ \xi\in \big[-\fl{\pi}{h}, \ \fl{\pi}{h}\big],
\end{equation}
where $i = \sqrt{-1}$, and the inverse semi-discrete Fourier transform is defined as 
\begin{equation}\label{inverse-semiFT1D}
    u_{j} = \mathcal{F}_{h,\,j}^{-1} [\check{u}] = \fl{1}{2\pi}\int_{-\pi/h}^{\pi/h} \check{u}(\xi)\,e^{i \xi x_{j}}\md \xi, \qquad \mbox{for} \ \  j \in \mathz.
\end{equation}
\end{definition}
Here, we use $\check{u}(\xi)$ to represent the semi-discrete Fourier transform of grid function $\bu$, distinguishing it from $\widehat{u}(\xi)$ -- the Fourier transform of continuous function $u(x)$. 
Different from the continuous Fourier transforms,  the Fourier space in the semi-discrete  transform is bounded, i.e., $\xi \in [-\pi/h, \pi/h]$. 
This can be explained by the aliasing formula (see Lemma \ref{lemma-aliasing}). 

\begin{definition}[Generalized hypergeometric function] 
For $p,\,q \in {\mathbb N}^0$, the generalized hypergeometric function is defined as
\bea\label{pFq}
\,_pF_q \big(a_1, \ldots, a_p;\, b_1, \ldots, b_q;\, z\big) = \sum_{k=0}^\infty \fl{(a_1)_k (a_2)_k \cdots (a_p)_k}{(b_1)_k (b_2)_k \cdots (b_q)_k} \fl{z^k}{k!},
\eea
where $a_l \in {\mathbb C}$ (for $1 \le l \le p$),    $b_m \in {\mathbb C}$ but $b_m \notin \big({\mathbb Z}^-\cup\{0\}\big)$ (for $1 \le m \le q$), and $(a)_k$ denotes the rising Pochhammer symbol, i.e.,
$${\displaystyle (a)_{k}={\begin{cases}1, &\mbox{for \ \ $k=0$}\\a (a+1) \cdots (a+k-1),\qquad \, & \mbox{for \ \ $k>0$}.\end{cases}}}$$
\end{definition} 
If $p \le q$, the series in (\ref{pFq}) is convergent for all values of $z \in {\mathbb C}$. 
If $p = q+1$, it converges for $|z|<1$. 
Particularly, if $p = 2$ and $q = 1$,  (\ref{pFq}) reduces to the well-known Gauss hypergeometric function $_2F_1$. 
For $p = q = 1$,   $_1F_1$ is often called as the confluent hypergeometric function.
For more discussion of the generalized hypergeometric functions, we refer the reader to references \cite{Prudnikov1986vol2, Samko2001}. 

Now we introduce our method.  
Denote $(-\Dt)_h^{\fl{\alpha}{2}}$ as the numerical approximation of the fractional Laplacian $(-\Dt)^{\fl{\alpha}{2}}$, and let $u_j = u(x_j)$ for $j \in \mathz$. 
At point $x = x_k$, we can approximate the Fourier transforms in definition (\ref{pseudo}) by the semi-discrete Fourier transforms and then obtain the approximation:
\bea\label{1Dscheme0}
(-\Dt)_h^{\fl{\alpha}{2}} u_k &=& \fl{1}{2\pi}\int_{-\pi/h}^{\pi/h} |\xi|^\ap \bigg(h\sum_{j = -\infty}^\infty u_j\,e^{-i\xi x_j}\bigg) e^{i\xi x_k} \md\xi\nn\\
&=& \fl{h}{\pi}\int_0^{\pi/h} \xi^\ap \bigg(\sum_{j = -\infty}^\infty u_j \cos\big(\xi(x_k-x_j)\big)\bigg) \md\xi\nn\\
&=& \sum_{j = -\infty}^\infty \bigg(\fl{h}{\pi}\int_0^{\pi/h} \xi^\ap \cos\big((k - j) h\,\xi \big)\md\xi\bigg)u_j,\qquad 
\mbox{for \ $k \in {\mathbb Z}$}.\qquad
\eea
It is clear that the evaluation of function $(-\Dt)^\fl{\ap}{2}u(x)$ at point $x_k$ depends on all  points $x_j \in {\mathbb R}$,  consistent with the nonlocality of  the fractional Laplacian $(-\Dt)^{\fl{\alpha}{2}}$.  
More precisely, the approximation in (\ref{1Dscheme0}) can be viewed as a weighted summation of $u_j$ for $j \in {\mathbb Z}$, and the weight coefficients depend on the distance between points $x_k$ and $x_j$. 
To calculate these coefficients, we first expand $\cos\big((k - j) h\,\xi \big)$ into the Taylor series and then integrate each term to obtain 
\beas
&&\og^{(1)}_{\ap,\,h}(k-j) := \fl{h}{\pi}\int_0^{\pi/h} \xi^\ap \cos\big((k - j) h\,\xi \big)\md\xi\\
&&\hspace{2cm}=\fl{h}{\pi}\int_0^{\pi/h} \xi^\ap \bigg(\sum_{n=0}^{\infty} (-1)^n \fl{[(k-j)h\,\xi\big]^{2n}}{(2n)!} \bigg)\md \xi \qquad\qquad\qquad \\
&&\hspace{2cm}= \Big(\fl{\pi}{h}\Big)^\ap\sum_{n=0}^{\infty} \fl{(-1)^n \big[\pi(k-j)\big]^{2n}}{(2n)! (2n+\ap+1)}, 
\qquad \mbox{for} \ \ j, k \in {\mathbb Z}.
\eeas
Note that the right-hand side can be reformulated in terms of the rising Pochhammer symbols, i.e.,
\bea\label{1Dcoeff}
&&\og^{(1)}_{\ap,\,h}(k-j)=\displaystyle \fl{\pi^\ap}{(\ap+1)h^{\ap}}\sum_{n=0}^{\infty} \fl{\big(\fl{\ap+1}{2}\big)_{n}}{\big(\fl{\ap+3}{2}\big)_{n} \big(\fl{1}{2}\big)_{n}\,n!} \bigg(\fl{-\pi^2 (k-j)^2}{4}\bigg)^n, \qquad\qquad\nn\\
&&\hspace{1.86cm}= \fl{\pi^\ap}{(\ap+1)h^{\ap}}\,_1F_2\Big(\fl{\ap+1}{2};\,\fl{\ap+3}{2}, \, \fl{1}{2};\, \fl{-\pi^2(k-j)^2}{4}\Big), 
\eea
for $j, k \in {\mathbb Z}$, by the definition of generalized hypergeometric function in (\ref{pFq}). 
It shows that the weight $\og^{(1)}_{\ap,\,h}$ is an even function of $(k-j)$, depending on parameters $\ap$ and $h$, especially  $\og_{\ap,\,h}^{(1)}(0) = {(\pi/h)^\ap}/{(\ap+1)}$. 
Substituting (\ref{1Dcoeff}) into  (\ref{1Dscheme0}), we obtain our numerical approximation to the one-dimensional fractional Laplacian as: 
\bea\label{1Dscheme}
(-\Dt)_h^{\fl{\alpha}{2}} u_k = \, \fl{\pi^\alpha}{(\alpha+1) h^\alpha} \sum_{j=-\infty}^{\infty} { }_1F_2\Big(\fl{\ap+1}{2};\, \fl{\ap+3}{2},\,\fl{1}{2};\, \fl{-\pi^2(k-j)^2}{4}\Big)\,u_j, \qquad 
\eea
for $k \in {\mathbb Z}$.  
{ Note that the same scheme is obtained in \cite{Huang0016} by matching the fractional Laplacian with a discrete operator on an infinite lattice.  }
Our scheme (\ref{1Dscheme}) holds for any $\ap \ge 0$, including $\ap = 2$.  
In the special cases of $\ap \in {\mathbb N}^0$, the coefficient function $\og_{\ap,\,h}^{(1)}$ in (\ref{1Dcoeff}) can be simplified to elementary functions. 
In fact,  for any $m \in {\mathbb N}$ and $n \in \mathz\backslash\{0\}$, the integral 
\bea\label{mncoeff1}
&&\int_0^{\pi/h} \xi^m \cos(nh \xi)\md\xi \nn\\
&&\hspace{0.5cm} =  \fl{m!}{(nh)^{m+1}} \bigg[\bigg(\sum_{k = 0}^{\lfloor(m-1)/2\rfloor}{(-1)^{n+k}}\fl{(n\pi)^{m-1-2k}}{(m-1-2k)!}\bigg) -{\rm mod}(m,2) \bigg].\qquad 
\eea
Choosing $\ap = m = 1$ or $2$,  we thus get the coefficients 
\beas
\og_{\ap,\,h}^{(1)}(k-j) =  \fl{1}{h^\ap (k-j)^{2}} \left\{\begin{array}{ll}
\displaystyle \big( (-1)^{k-j} - 1\big)/\pi,  &\ \mbox{if} \ \ \ap = 1, \\
\displaystyle 2(-1)^{k-j},       &\ \mbox{if} \ \ \ap = 2, \\
\end{array}\right. \quad \ \mbox{for}\ \, (k-j) \in \mathz\backslash\{0\}. 
\eeas
If $m = 0$,  the integral in (\ref{mncoeff1}) is zero for any $n \in \mathz\backslash\{0\}$, and thus the scheme (\ref{1Dscheme0}) reduces to the identity operator for $\ap = 0$. 
It is consistent with the analytical property of the fractional Laplacian. 

\begin{lemma}\label{remark2p1}
The discretized operator $(-\Dt)_h^\fl{\ap}{2}$ from our scheme (\ref{1Dscheme}) has exactly the same symbol $|\xi|^\ap$ as the fractional Laplacian $(-\Dt)^\fl{\ap}{2}$, i.e., 
\begin{equation}\label{remark2p1-eq}
    {\mathcal F_h}\big[(-\Dt)^\fl{\ap}{2}_h \bu\big] = |\xi|^\ap \mathcal{F}_{h}[\bu], \qquad\mbox{for} \ \ \ap > 0,
\end{equation}
independent of mesh size $h$. 
\end{lemma}

\begin{proof} 
By scheme (\ref{1Dscheme}) and the definition of semi-discrete Fourier transform, we obtain
\bea\label{eq2-2-0}
{\mathcal F_h}\big[(-\Dt)^\fl{\ap}{2}_h \bu\big] 
    &=& h\sum_{k=-\infty}^{\infty}\bigg( \sum_{j=-\infty}^{\infty} \og_{\ap,\,h}^{(1)}(k-j)\,u_{j}\bigg) e^{-i\xi x_{k}}\nn \\
    &=& h\sum_{j=-\infty}^{\infty} u_{j}\bigg( \sum_{k=-\infty}^{\infty}\og_{\ap,\,h}^{(1)}(k-j) e^{-i\xi x_{k-j}}\bigg) e^{-i\xi x_{j}}\qquad \quad\nn\\
    &=& h\sum_{j=-\infty}^{\infty} u_{j} \Big(\fl{1}{h}\mathf \big[\boldsymbol\og_{\ap,\,h}\big]\Big) e^{-i\xi x_{j}},
\eea
where we denote $\boldsymbol\og_{\ap,\,h} = \big\{\og^{(1)}_{\ap,\,h}(k)\big\}_{k\in \mathz}$. 
The definition of $\og_{\ap,\,h}^{(1)}$ shows that
\beas
\og_{\ap,\,h}^{(1)}(k)=\fl{h}{2\pi}\int_{-\pi/h}^{\pi/h}|\xi|^\ap e^{i\xi x_{k}}\md\xi = h{\mathcal F}_{h, k}^{-1}\big[|\xi|^\ap\big], \qquad \mbox{for}\ \, k \in {\mathz},
\eeas
which implies that $\mathf \big[\boldsymbol\og_{\ap,\,h}\big] = h |\xi|^\ap$.  
Substituting it into (\ref{eq2-2-0}) and using the definition of the semi-discrete Fourier transform again immediately yields (\ref{remark2p1-eq}), which holds for any mesh size  $h$. 
\end{proof}

\begin{remark}\label{lemmeA}
Let $(\widetilde{-\Dt})^\fl{\ap}{2}_h$ represent the finite difference approximation of the fractional Laplacian  in \cite{Zoia2007,  Hao2021}. 
Then it satisfies 
\begin{equation}\label{remark2p2-eq}
    {\mathcal F_h}\big[(\widetilde{-\Dt})^\fl{\ap}{2}_h \bu\big] = \big[|\xi|^\ap + {\mathcal O}(|\xi|^{2+\ap}h^2)\big]\mathcal{F}_{h}[\bu],
\end{equation}
for small mesh size $h > 0$. 
It suggests that the approximation $(\widetilde{-\Dt})^\fl{\ap}{2}_h$  has different symbol from the fractional Laplacian $(-\Dt)^\fl{\ap}{2}$. But,  its leading order term is $|\xi|^\ap$ if mesh size $h$ is small. 
\end{remark}

Lemma \ref{remark2p1} suggests that our scheme (\ref{1Dscheme}) can be viewed as a discrete pseudo-differential operator with symbol $|\xi|^\ap$ -- an exact discrete analogue of the Laplace operator $(-\Dt)^\fl{\ap}{2}$ for $\ap > 0$.  
This unique property of our method distinguishes it from other existing methods of the fractional Laplacian. 
Note that the formulation and implementation of our method are  similar to the finite difference methods \cite{Duo2018, Huang2014, Minden2020, Hao2021}, but our method can achieve significantly higher accuracy. 
Moreover, our method has much lower computational cost when assembling the stiffness matrix, particularly in high dimensions. 
More numerical comparison can be found in Section \ref{section5-1}.

\section{Error analysis}
\label{section3}
\setcounter{equation}{0}

In this section, we first study the numerical errors of our method in discretizing the fractional Laplacian $(-\Dt)^\fl{\ap}{2}$, and detailed error estimates are provided under different conditions of function $u$. 
Then the stability and  convergence of our method in solving the fractional Poisson equations are discussed in Section \ref{section3-1}. 

Let $C^{p, \gm}({\mathbb R})$ denote the H\"older space, for $p \in {\mathbb N}^0$ and $\gm\in(0, 1]$. 
First, we introduce the following lemmas on the continuous and semi-discrete Fourier transforms \cite{Trefethen2000, Nissila2021}:
\begin{lemma}
\label{lemma-uhat-decay}
Suppose $u\in L^2(\mathr)$, and  $\widehat{u}$ denotes its Fourier transform.  \vspace{1mm}

\begin{enumerate}
\item[(i)]  Suppose $u\in C^{p,\gamma}(\mathr)$ for $p\in\mathbb{N}^0$ and $\gamma\in(0,1]$. 
    Furthermore, if $u^{(k)}\in L^2(\mathr)$ for $k\leq p-1$, and $u^{(p)}$ has bounded variation, then there is
   \begin{equation}
        \widehat{u}(\xi) = \mathcal{O}\big(|\xi|^{-(p+1+\gamma)}\big), \qquad   \text{as }\,   |\xi|\to\infty. 
    \end{equation}
    
\item[(ii)]  If $u\in C^{\infty}(\mathr)$,  and $u^{(k)}\in L^2(\mathr)$ for $k \in {\mathbb N}$, then there is 
    \begin{equation}
        \widehat{u}(\xi) = o\big(|\xi|^{-m}\big), \qquad  \text{as }\,  |\xi|\to \infty,
    \end{equation}
for any $m \ge 0$, and the converse also holds.     
\end{enumerate} 
\end{lemma}
Lemma \ref{lemma-uhat-decay} shows that the smoother the function $u$ is, the faster the Fourier transform $\widehat{u}$ decays. 
In the following, we denote ${\bf v}$ as a grid function on $h{\mathz}$ with $v_j = u(x_j)$ for $j \in \mathz$, and $\check{v}(\xi)$ represents its semi-discrete Fourier transform.

\begin{lemma} (Aliasing formula)
\label{lemma-aliasing}
Suppose $u\in L^2(\mathr)$ has a first derivative of bounded  variation, and $\widehat{u}$ denotes its Fourier transform. 
Then there is
\begin{equation}\label{aliasing}
    \check{v}(\xi) = \sum_{j=-\infty}^{\infty} \widehat{u}\big(\xi + \fl{2 j\pi}{h}\big), \qquad \text{for }\,  \xi\in\big[-\fl{\pi}{h},\, \fl{\pi}{h}\big],
\end{equation}
for any $h > 0$.
\end{lemma}

From Lemmas \ref{lemma-uhat-decay}--\ref{lemma-aliasing}, we immediately obtain the following results. 
\begin{lemma}\label{lemma-uhat-decay-utilde}
Suppose $u\in L^2(\mathr)$ has a first derivative of bounded variation, and $\widehat{u}$ is its Fourier transform.  
\begin{enumerate}
\item[(i)]    Let $p\in\mathbb{N}^0$ and $\gamma\in(0,1]$.  
    If $u\in C^{p,\gamma}(\mathr)$, $u^{(k)}\in L^2(\mathr)$ for $k\leq p-1$, and $u^{(p)}$ has bounded variation, then there is
    \begin{equation}\label{eq2-8-2}
        \big| \check{v}(\xi) - \widehat{u}(\xi) \big| = \mathcal{O}(h^{p+1+\gamma}), \qquad \text{as }\, h\to 0,
    \end{equation}
for  $\xi \in \big[-{\pi}/{h}, \,{\pi}/{h}\big]$. 
\item[(ii)] If $u\in C^{\infty}(\mathr)$,  and $u^{(k)}\in L^2(\mathr)$ for $k \in {\mathbb N}$, then there is 
    \begin{equation}\label{eq3-11-1}
        \big| \check{v}(\xi) - \widehat{u}(\xi) \big| = o(h^{m}), \qquad \text{as } \, h\to 0,
    \end{equation}
    for any $m \ge 0$ and $\xi \in \big[-{\pi}/{h}, \,{\pi}/{h}\big]$. 
\end{enumerate} 
\end{lemma}
\begin{proof} 
Let's focus on the proof of (\ref{eq2-8-2}). 
Using the aliasing formula (\ref{aliasing}) and then the triangle inequality, we get
\beas
    \big|\check{v}(\xi) - \widehat{u}(\xi)\big| &=& \bigg|\sum_{j=1}^{\infty} \Big[\widehat{u}\big(\xi-\fl{2j\pi}{h}\big)+\widehat{u}\big(\xi+\fl{2j\pi}{h}\big) \Big]\bigg|\nn\\
&\le& \sum_{j=1}^{\infty} \Big(\big|\widehat{u}\big(\xi-\fl{2j\pi}{h}\big)\big|+\big|\widehat{u}\big(\xi+\fl{2j\pi}{h}\big)\big| \Big),  \quad \ \ \mbox{for} \ \, \xi \in \big[-\fl{\pi}{h},\, \fl{\pi}{h}\big].
\eeas
Lemma \ref{lemma-uhat-decay} (i) shows that $|\widehat{u}(\xi)| \leq C|\xi|^{-(p+1+\gamma)}$ as $|\xi|\to\infty$. From it, we can further obtain: 
\beas
\big| \check{v}(\xi) - \widehat{u}(\xi) \big| &\leq& C \sum_{j=1}^{\infty} \Big[\big(2j+1\big)\fl{\pi}{h}\Big]^{-(p+1+\gamma)} \\
&=& C h^{p+1+\gamma}\sum_{j=1}^{\infty}\fl{1}{(2j+1)^{p+1+\gamma}} \leq C h^{p+1+\gamma}, \qquad\mbox{as\ \ $h \to 0$}
\eeas
where the constant $C > 0$ is independent of $h$. 
The proof of (\ref{eq3-11-1}) can be done by following the similar arguments, which we will omit for brevity. 
\end{proof}

For grid function $\bu$, define the norm $\|{\bf u}\|_{l^\infty} = \sup_{j \in {\mathbb Z}} |u_j|$. 
Then we present the error estimates of our method in the following theorem.  
\begin{theorem}\label{thm1}
Suppose $u\in L^2(\mathr)$. 
Denote  ${\bf u}$ as a grid function with $u_j = u(x_j)$ for $j \in \mathz$. 
Let  $(-\Dt)^{\fl{\ap}{2}}_{h}$ represent the numerical approximation of the fractional Laplacian, as defined in (\ref{1Dscheme0}).  \vspace{1mm}

\begin{enumerate}
\item[(i)]   Let $p\in\mathbb{N}^0$ and $\gamma\in(0,1]$.  
 If $u\in C^{p,\gamma}(\mathr)$, $u^{(k)}\in L^2(\mathr)$ for $k\leq p-1$, and $u^{(p)}$ has bounded variation, then there is 
\bea\label{LTE1}
&&\big\|(-\Dt)^{\fl{\ap}{2}} \bu - (-\Dt)_h^\fl{\ap}{2}\bu\big\|_{l^\infty}  \,\leq\, C h^{p+\gamma-\alpha}, \\
\label{LTE1-1}
&&\big\|(-\Dt)^{\fl{\ap}{2}} \bu - (-\Dt)_h^\fl{\ap}{2}\bu\big\|_{l^2}  \,\leq\, C h^{p+\gamma-\alpha+1/2}\qquad\qquad\quad\quad
\eea
with $C$ a positive constant independent of $h$.  \vspace{1mm} 

\item[(ii)]  If $u\in C^{\infty}(\mathr)$,  and $u^{(k)}\in L^2(\mathr)$ for $k \in {\mathbb N}$, then there is 
\bea\label{LTE2}
\big\|(-\Dt)^{\fl{\ap}{2}} \bu - (-\Dt)_h^\fl{\ap}{2}\bu\big\| <  C h^m, \qquad \mbox{for any} \ \, m \ge 0,
\eea
which holds for both $l^\infty$- and $l^2$-norm. 
\end{enumerate} 
\end{theorem}
Theorem \ref{thm1} suggests that the accuracy of our method depends on the smoothness of function $u$ -- the smoother the function $u$, the higher the accuracy. 
If $u \in C^\infty({\mathbb R})$,  our method has a spectral accuracy. 
In the following, we will provide the proof mainly for  (\ref{LTE1}) and (\ref{LTE1-1}), while the proof of (\ref{LTE2}) can be done by following the similar lines.  
\begin{proof}
Note that $u_k = u(x_k)$ for $k \in {\mathbb Z}$.  
By the definition of  $(-\Dt)^\fl{\ap}{2}$ in (\ref{pseudo}) and the scheme of $(-\Dt)_h^\fl{\ap}{2}$ in (\ref{1Dscheme0}), we obtain 
\beas
&&\big|(-\Dt)^\fl{\ap}{2} u(x_k) - (-\Dt)_h^\fl{\ap}{2}u_k\big|\\
&&\hspace{2.cm}=\fl{1}{2\pi}\bigg|\int_{-\infty}^{\infty}|\xi|^\ap\widehat{u}(\xi)\,e^{i\xi x_k}\md\xi - \int_{-\pi/h}^{\pi/h}|\xi|^\ap \check{u}(\xi)\,e^{i\xi x_k}\md\xi \bigg| \\
&&\hspace{2.cm}\leq \fl{1}{2\pi} \int_{-\pi/h}^{\pi/h}|\widehat{u}(\xi) - \check{u}(\xi)||\xi|^\ap\md\xi + \fl{1}{2\pi}\int_{|\xi|>\fl{\pi}{h}}|\widehat{u}(\xi)| |\xi|^\ap\md \xi\\
&&\hspace{2.cm}= I + II, \qquad \mbox{for} \ \ k \in \mathz. \qquad\quad
\eeas
Now we estimate  terms $I$ and $II$ separately. 
For term $I$,  we use Lemma \ref{lemma-uhat-decay-utilde} (i) and obtain
\beas
&&I := \fl{1}{2\pi} \int_{-\pi/h}^{\pi/h}|\widehat{u}(\xi) - \check{u}(\xi)||\xi|^\ap\md\xi  \leq C h^{p+1+\gamma} \int_{-\pi/h}^{\pi/h}|\xi|^\ap \md\xi \, \leq \, C h^{p+\gamma-\ap}
\eeas
with constant $C > 0$ independent of $h$. 
While using Lemma \ref{lemma-uhat-decay} (i) to term $II$ yields
\beas
&&II := \fl{1}{2\pi}\int_{|\xi|>\fl{\pi}{h}}|\widehat{u}(\xi)| |\xi|^\ap\md \xi \le C \int_{|\xi|>\fl{\pi}{h}} |\xi|^{\ap - (p + 1 +\gm)}\md \xi \, \leq \, C h^{p+\gamma-\ap}. \qquad\quad
\eeas
Combining the estimates of terms $I$ and $II$ immediately yields the result in (\ref{LTE1}). 

Next, we prove  (\ref{LTE1-1}).  
Using the definition  in (\ref{pseudo}) and the scheme  in (\ref{1Dscheme0}), we get
\bea\label{l2-eq1}
&&\big\|(-\Dt)^{\fl{\ap}{2}}\bu - (-\Dt)^{\fl{\ap}{2}}_{h}\bu\big\|^2_{l^2} \nn\\
&&\hspace{1cm} = h\sum_{k\in\mathz} \Big| (-\Dt)^{\fl{\ap}{2}}u(x_k) - (-\Dt)^{\fl{\ap}{2}}_{h}u_k \Big|^2\nn\\
&&\hspace{1cm} = \fl{h}{4\pi^2}\sum_{k\in\mathz} \bigg|\underbrace{\int_{|\xi|>\fl{\pi}{h}} |\xi|^{\ap} \widehat{u}(\xi) e^{i\xi x_k}\md \xi}_{I_{1,k}} + \underbrace{\int_{-\pi/h}^{\pi/h} |\xi|^{\ap} \big[\widehat{u}(\xi)-\check{u}(\xi)\big] e^{i\xi x_k}\md \xi}_{I_{2,k}}\bigg|^2\nn\\
&&\hspace{1cm} = \fl{h}{4\pi^2}\sum_{k\in\mathz} \Big(|I_{1,k}|^2 + |I_{2,k}|^2 + I_{1,k}\cdot\bar{I}_{2,k} + \bar{I}_{1,k}\cdot{I}_{2,k}\Big).
\eea
For term $|I_{1,k}|^2$, we obtain
\bea
&&\sum_{k\in\mathz} |I_{1,k}|^2 := \sum_{k\in\mathz} \bigg(\int_{|\xi| >\fl{\pi}{h}}\widehat{u}(\xi) |\xi|^{\ap} e^{i\xi x_k}\md \xi\bigg)\bigg(\int_{|\zeta| > \fl{\pi}{h}} \overline{{\widehat{u}}(\zeta) |\zeta|^{\ap} e^{i\zeta x_k}}\md \zeta\bigg)\nn\\
&&\hspace{17mm} = \int_{|\xi| >\fl{\pi}{h}}\widehat{u}(\xi) |\xi|^{\ap} \int_{|\zeta| >\fl{\pi}{h}}\overline{\widehat{u}}(\zeta) |\zeta|^{\ap}  \bigg(\sum_{k\in\mathz} \, e^{i(\xi-\zeta)x_k}\bigg) \md\xi \md\zeta.\nn
\eea
Noticing that the  Dirac delta function $\dt(\xi-\zeta) = ({h}/{2\pi})\sum_{k\in {\mathbb Z}} e^{i(\xi-\zeta)x_k}$, we then further obtain 
\bea\label{term1}
&&\sum_{k\in\mathz} |I_{1,k}|^2  = \fl{2\pi}{h} \int_{|\xi| > \fl{\pi}{h}} |\widehat{u}(\xi) |^2 |\xi|^{2\ap}\md \xi \nn\\
&&\hspace{17mm} \leq \fl{C}{h} \int_{|\xi| > \fl{\pi}{h}}|\xi|^{2(\ap-p-\gm-1)}\md\xi  \ \leq \ Ch^{2(p+\gamma-\ap)}, \qquad\quad
\eea
by Lemma \ref{lemma-uhat-decay} (i).
Following the similar lines and using Lemma \ref{lemma-uhat-decay-utilde} (i), we get
\bea\label{term2}
&&\sum_{k\in\mathz} |I_{2,k}|^2  = \fl{2\pi}{h} \bigg(\int_{-\pi/h}^{\pi/h} \big|\widehat{u}(\xi)-\check{u}(\xi)\big|^2 |\xi|^{2\ap}\md \xi\bigg)\nn\\
&&\hspace{17mm} \le Ch^{2(p+\gamma)+1} \int_{-\pi/h}^{\pi/h}|\xi|^{2\ap}\md\xi \ \leq \ Ch^{2(p+\gamma-\ap)}. \qquad\quad \
\eea
The estimate of term $I_{1,k}\cdot \bar{I}_{2,k}$ can be obtained by first using the Cauchy--Schwarz inequality and then (\ref{term1})--(\ref{term2}), i.e., 
\bea\label{term3}
\sum_{k\in\mathz} I_{1,k}\cdot \bar{I}_{2,k} &\le& \Big(\sum_{k\in\mathz} |I_{1,k}|^2\Big)^\fl{1}{2}\Big(\sum_{k\in\mathz} |I_{2,k}|^2 \Big)^\fl{1}{2} \ \leq \ Ch^{2(p+\gamma-\ap)}. \qquad\quad \
\eea
We can similarly obtain the estimates of term $\bar{I}_{1,k}\cdot I_{2,k}$. 
Substituting the estimates of the four terms in (\ref{l2-eq1}) and after simple calculation, we can immediately obtain the result in (\ref{LTE1-1}).
\end{proof}

Theorem \ref{thm1} provides the error estimates of our method in approximating the fractional Laplacian over ${\mathbb R}$. 
If a bounded domain $\Og \subset {\mathbb R}$ is considered, we introduce the norms
\bea\label{normOg}
\|{\bf u}\|_{l^\infty(\Og)} = \max_{j \in \Og_h} | u_j |, \qquad
\|{\bf u}\|_{l^2(\Og)} = \Big(h\sum_{j \in \Og_h} | u_j |^2\Big)^{1/2},
\eea 
and the inner product 
\beas
\langle \bu, \, \bv\rangle_{\Og} : = h \sum_{j\in\Og_h} u_j \bar{v}_j.
\eeas
Here,  the index set is defined as ${\Og}_{h} = \big\{j\, |\, j\in\mathbb{Z},\, \text{ and } x_j\in\Og\big\}$. 
It is obvious that under the same conditions, the estimates in Theorem \ref{thm1} also hold if the norm over $h{\mathbb Z}$ (i.e., $l^\infty$ or $l^2$) is replaced with the norms in (\ref{normOg}) on $\Og_h$. 

\subsection{Stability and convergence}
\label{section3-1}

The fractional Poisson equation is one important building block in the study of  nonlocal/fractional  PDEs. 
It has been widely studied  and often used as the benchmark to test numerical methods for the fractional Laplacian \cite{Huang2014, Duo2018, Minden2020, Hao2021}. 
Here, we consider the fractional Poisson equation of the form \cite{Acosta2017, Ros-Oton2014}:
\bea\label{Poisson}
\begin{aligned}
(-\Dt)^\fl{\ap}{2}u(x) = f(x), &\qquad  \mbox{for} \ \ x\in\Og, \\
u(x) = 0,  & \qquad  \mbox{for} \ \ x\in\Og^c,
\end{aligned}
\eea
where $\Og^c = {\mathbb R}\backslash\Og$. 
In the following, we study the stability and convergence of our method in solving  (\ref{Poisson}). 
The direct application of our method to (\ref{Poisson}) yields the system of difference equations: 
\bea\label{Poisson-discrete}
(-\Dt)_h^{\fl{\ap}{2}}u_j^h = f(x_j), &\quad \mbox{for} \ \,  j\in {\Og}_{h}, \\
\label{BC-discrete}
u_j^h = 0, &\quad \mbox{for} \ \,  j\in {\Og}_{h}^c,
\eea
where $u_j^h$ represents the numerical approximation to solution $u(x_j)$.

\begin{lemma}[Parseval's identity]\label{lemma-parseval-ident}
For grid function ${\bf u}, {\bf v} \in l^2({\mathbb R})$,  there is
\bea\label{Parseval}
{ \langle{\bf u}, \, {\bf v}\rangle_h} = \fl{1}{2\pi}\int_{-\pi/h}^{\pi/h} \check{u}(\xi)\,\overline{\check{v}}(\xi)\md\xi.
\eea
\end{lemma}
It can be proved by using the definition of the semi-discrete Fourier transforms.

\begin{lemma}\label{lemma-flh-bound}
Suppose $\bu_h = \big\{u_j^h\big\}_{j \in \mathz}$ is the solution of the discrete problem (\ref{Poisson-discrete})--(\ref{BC-discrete}). 
Then there is 
\begin{equation}\label{inq-flh}
\|{\bf u}_h\|_{l^2} \leq C \|(-\Dt)^\fl{\ap}{2}_{h}{\bf u}_h\|_{{l^2}(\Og)}
\end{equation}
with $C$ a positive constant independent of $h$. 
\end{lemma}
\begin{proof}
By Parseval's identity, we obtain
\beas
    \|\bu_h\|_{l^2}^2 &=& \fl{1}{2\pi}\int_{-\pi/h}^{\pi/h} |\check{u}(\xi)|^2\md\xi\\
    &=& \fl{1}{2\pi}\bigg(\int_{-\varepsilon}^{\varepsilon} |\check{u}(\xi)|^2\md\xi + \int_{\varepsilon<|\xi|<\fl{\pi}{h}} |\check{u}(\xi)|^2\md\xi\bigg),
\eeas
where $\veps > 0$ will be discussed later.
Due to the homogeneous boundary conditions in (\ref{BC-discrete}), we get 
\begin{equation*}
\big|\check{u}(\xi)\big| = \Big| h\sum_{j \in {\Og_h}} u_j e^{-i\xi x_j} \Big| \leq h\sum_{j \in {\Og_h}} |u_j| \leq 
\sqrt{|\Og|}\|\bu_h\|_{l^2},
\end{equation*}
by the triangle and H\"{o}lder's inequalities. 
It immediately leads to the estimate
\beas\label{estimate1}
\int_{-\varepsilon}^{\varepsilon} |\check{u}(\xi)|^2\md\xi \,\leq\, \int_{-\varepsilon}^{\varepsilon} |\Og|\|\bu_h\|^2_{l^2}\md\xi = \big(2 \varepsilon |\Og|\big)\|\bu_h\|^2_{l^2}.
\eeas
Choose $\varepsilon$ such that $2\varepsilon|\Og|<\fl{1}{2}$, we have
\bea\label{estimate1-1}
\int_{-\varepsilon}^{\varepsilon} |\check{u}(\xi)|^2\md\xi \,\leq\, \fl{1}{2}\|\bu_h\|^2_{l^2}.
\eea

On the other hand, we have
\beas
&&\int_{\varepsilon<|\xi|<\fl{\pi}{h}} |\check{u}(\xi)|^2\md\xi = \int_{\varepsilon<|\xi|<\fl{\pi}{h}} \fl{1}{|\xi|^{\ap}} \Big({|\xi|^{\ap}|\check{u}(\xi)|^2}\Big)\md\xi \nn\\
&&\hspace{2.9cm} \le \, \varepsilon^{-\ap}\int_{\varepsilon<|\xi|<\fl{\pi}{h}} {|\xi|^{\ap}|\check{u}(\xi)|^2}\md\xi  \nn \\
&&\hspace{2.9cm}  \le \, \varepsilon^{-\ap} \int_{|\xi|<\fl{\pi}{h}} {|\xi|^{\ap}|\check{u}(\xi)|^2}\md\xi \, = 2\pi \varepsilon^{-\ap} {\big \langle \bu_h, (-\Dt)^{\fl{\ap}{2}}_{h} \bu_h\big\rangle_h}
\eeas
by the Parseval's identity.  
Noticing the homogeneous Dirichlet boundary conditions in \eqref{BC-discrete}, we further obtain 
\bea\label{estimate2}
&&\int_{\varepsilon<|\xi|<\fl{\pi}{h}} |\check{u}(\xi)|^2\md\xi \, \le\,  2\pi \varepsilon^{-\ap} \big\langle \bu_h, (-\Dt)^{\fl{\ap}{2}}_{h} \bu_h\big\rangle_{\Og} \nn\\
&&\hspace{3.2cm}  \le \,  2\pi \varepsilon^{-\ap} \|\bu_h\|_{l^2(\Og)}\|(-\Dt)^{\fl{\ap}{2}}_{h} \bu_h\|_{l^2(\Og)}.
\eea
Combining the estimates in \eqref{estimate1-1} and \eqref{estimate2} and noticing $\|\bu_h\|_{l^2} = \|\bu_h\|_{l^2(\Og)}$ for solution of (\ref{Poisson-discrete})--(\ref{BC-discrete}), we immediately obtain \eqref{inq-flh}, where $C$ depends on $\varepsilon$ but independent of $h$.
\end{proof}

\begin{theorem}[Stability]
\label{thm2}
Suppose $\bu_h = \big\{u_j^h\big\}_{j \in \Og_h}$ is the solution of the discrete problem (\ref{Poisson-discrete})--(\ref{BC-discrete}). 
Then it satisfies 
\begin{equation}\label{inq-stability}
    \|\bu_h\|_{l^2(\Og)}  \leq  C \|{\bf f}\|_{l^2(\Og)},
\end{equation}
where  $C$ is a positive constant independent of $h$, and ${\bf f} = \big\{f(x_j)\big\}_{j \in \Og_h}$. 
\end{theorem}
\begin{proof}
Multiplying both sides of (\ref{Poisson-discrete}) with $(-\Dt)^{\fl{\alpha}{2}}_{h}\bar{u}_j^h$ and summing it for index $j \in \Og_h$, we get
\beas
\big\langle(-\Dt)^{\fl{\alpha}{2}}_{h}\bu_h,\, (-\Dt)^{\fl{\alpha}{2}}_{h}\bu_h \big\rangle_{{\Og}} = \big\langle {\bf f},\, (-\Dt)^{\fl{\alpha}{2}}_{h}\bu_h \big\rangle_{{\Og}}.
\eeas
We then use the Cauchy--Schwarz inequality and obtain 
\bea\label{thm2-eq1}
\big\|(-\Dt)^{\fl{\alpha}{2}}_{h}\bu_h \big\|_{l^{2}(\Og)} \leq  \sqrt{|\Og|}\,\|{\bf f}\|_{l^2(\Og)}  .
\eea
From (\ref{inq-flh}) and (\ref{thm2-eq1}), we then obtain
\beas
\|\bu_h\|_{l^2} = \|\bu_h\|_{l^2(\Og)} \leq C\|(-\Dt)^{\fl{\ap}{2}} \bu_h\|_{l^2(\Og)} \leq C\|{\bf f}\|_{l^2(\Og)}.
\eeas
with $C > 0$ independent of $h$.
\end{proof}
The result in (\ref{inq-stability}) implies that if $f(x) = 0$, then $\|{\bf u}\|_{l^{2}(\Og)}=0$. 
Therefore, we obtain $u(x) \equiv  0$ for $x \in \Og$, which implies that there exists a unique solution to (\ref{Poisson-discrete})--(\ref{BC-discrete}). 

\begin{theorem}[Convergence]
\label{thm3}
Suppose $u(x)$ is the exact solution of  (\ref{Poisson}), and denote $\bu = \{u(x_j)\}_{j\in {\Og_h}}$.  Let $\bu_h$ be the solution of the discrete system (\ref{Poisson-discrete})--(\ref{BC-discrete}). 
 \vspace{1mm}

\begin{enumerate}
\item[(i)]   Suppose $u \in C^{p,\gamma}(\bar{\Og})$ for $p \in {\mathbb N}^0$ and $\gm \in (0, 1]$.  
Moreover, if $u^{(k)}\in L^2(\bar{\Og})$ for $k\leq p-1$, and $u^{(p)}$ is bounded variation, then
\begin{equation}\label{ineq-sol1}
        \|\bu - \bu_h\|_{l^{2}(\Og)} \,\leq\, C h^{p+\gamma-\alpha+1/2}.
\end{equation}

\item[(ii)]  If $u\in C^{\infty}(\bar{\Og})$ and $u^{(k)}\in L^2(\bar{\Og})$, then there is
 \begin{equation}\label{ineq-sol2}
        \|\bu - \bu_h\|_{l^{2}(\Og)}\, < \, C h^m, \qquad \mbox{for\, any} \ \ m\ge 0,
  \end{equation}
that is, our method has a spectral accuracy.
\end{enumerate}
\end{theorem}

\begin{proof}
Denote the grid error function  ${\bf e} = \{e_j\}_{j \in {\mathbb Z}}$ with 
\beas
 e_j = u(x_j) - u_j^h, \qquad \mbox{for} \ \, j \in {\mathbb Z}. 
\eeas
From (\ref{Poisson}) and (\ref{Poisson-discrete})--(\ref{BC-discrete}), we obtain  
\beas
\begin{aligned}
(-\Dt)^{\fl{\ap}{2}}_{h} e_j = (-\Dt)^{\fl{\ap}{2}}u(x_j) -  (-\Dt)^{\fl{\ap}{2}}_{h} u_j^h, &\quad& \mbox{for} \ \  j \in \Og_h ,\qquad\quad\\
e_j = 0, &\quad\ & \mbox{for} \ \  j \in \Og_h^c. \qquad\quad\\
\end{aligned}
\eeas
Hence, following the similar lines as in proving Theorem \ref{thm2}, we obtain
\beas
\|{\bf e}\|_{l^2(\Og)} &\le&C  {\|(-\Dt)^{\fl{\ap}{2}}{\bf u} -  (-\Dt)^{\fl{\ap}{2}}_{h}{\bf u}_h\|_{l^2(\Og)}}\\
&\le&C  {\|(-\Dt)^{\fl{\ap}{2}}{\bf u} -  (-\Dt)^{\fl{\ap}{2}}_{h}{\bf u}_h\|_{l^2}}.
\eeas
Combining it with the  estimates in (\ref{LTE1-1}) and (\ref{LTE2}), we  obtain (\ref{ineq-sol1}) and (\ref{ineq-sol2}), respectively.
\end{proof}

 \begin{remark}
Our numerical studies show that the observed accuracy rate could be much higher than that predicted in (\ref{ineq-sol1}); see Figures \ref{fig:ex00} and \ref{fig:ex4}. 
Particularly, for the benchmark fractional Poisson problem (i.e., $s = \ap/2$ in (\ref{elliptic-ex1})), our method has the numerical error of ${\mathcal O}(h^{\min\{1, (\ap+1)/2\}})$; see  Figure \ref{fig:ex00}. 
\begin{figure}[htb!]
\centerline{
\includegraphics[width=5.6cm, height = 4.86cm]{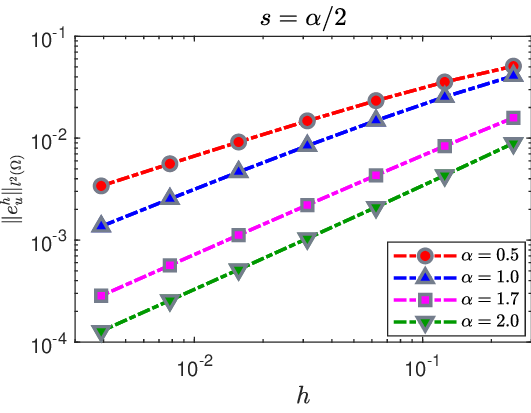}}
\caption{Numerical errors  of our method in solving the benchmark Poisson problem (\ref{elliptic-ex1}) with $s = \ap/2$, where the exact solution $u\in C^{0,\fl{\ap}{2}}(\bar{\Og})$.}
\label{fig:ex00}
\end{figure}
In this case, the exact solution $u\in C^{0,\fl{\ap}{2}}(\bar{\Og})$, and our method has the same accuracy as the finite element/difference methods in \cite{Acosta2017short, Acosta2017, Hao2021, Duo2018} since the low regularity of solution at the boundary becomes dominant.  
\end{remark}

\section{Generalization to high dimensions}
\label{section4}
\setcounter{equation}{0}

{Compared to one-dimensional cases, it is more challenging to approximate the fractional Laplacian in high dimensions. 
In this section, we generalize our spectral method for the one-dimensional fractional Laplacian to high-dimensional cases with $d \ge 2$.   
Specifically, we will present two generalization approaches in Sections \ref{section4-1} and \ref{section4-2}, with their implementation and performance comparison discussed in Section \ref{section4-3}.}

Let \ ${\bf j}  = (j_1, \, j_2,\,\cdots,j_d) \in {\mathbb Z}^d$.   
Define the $d$-dimensional grid points $\bx_{\bf j} = {\bf j} h = (j_1h,\, j_2 h,\, \cdots, \, j_d h)$, and denote  grid function $\bu^{(d)}=\{u_{\bj}\}_{\bj\in\mathz^d}$. 
In high dimensions, the inner product and associated norm can be defined similarly as in Section \ref{section2}. 
First, we generalize the definition of the semi-discrete Fourier transform to $d$ dimensions.
\begin{definition}[$d$-dimensional semi-discrete Fourier transform]
\label{Def-SFT-d}
For a grid function $\bu^{(d)}\in l^2({\mathbb R}^d)$, the semi-discrete Fourier transform ${\mathcal F}_h$ is defined as
\begin{equation}\label{semiFT}
    \check{u}(\bxi) = \big(\mathcal{F}_h [\bu^{(d)}]\big)(\bxi) = h^d\sum_{\bj\in\mathz^d} u_{\bj}\,e^{-i\bxi \cdot \bx_{\bj}}, \quad \ \mbox{for} \ \ \bxi\in \big[-\fl{\pi}{h},\ \fl{\pi}{h}\big]^d, 
\end{equation}
where $\bxi \cdot \bx_{\bf j}$ denotes the dot product of vectors $\bxi$ and $\bx_{\bf j}$. 
The inverse semi-discrete Fourier transform is defined by
\begin{equation}\label{inverse-semiFT}
    u_{\bj} = \mathcal{F}_{h,\,\bf j}^{-1} [\check{u}] = \fl{1}{(2\pi)^d} \underbrace{\int_{-\pi/h}^{\pi/h}\cdots\int_{-\pi/h}^{\pi/h}}_{d \, {\rm times}} \check{u}(\bxi) \,e^{i\bxi \cdot \bx_{\bj}} \md \bxi, \quad \ \mbox{for} \ \ \bj \in \mathz^d.
\end{equation}
\end{definition}

{In high dimensions ($d > 1$), we can generalize the scheme (\ref{1Dscheme0}) from two different aspects.}

{\subsection{Generalization I}
\label{section4-1}

Following the same approach as in obtaining (\ref{1Dscheme0}), we can approximate the $d$-dimensional Fourier transforms in definition (\ref{pseudo}) using the semi-discrete Fourier transforms and then obtain the approximation at point $\bx =\bx_{\bf k}$ as:
\beas
&&(-\Dt)^{\fl{\ap}{2}}_{h} u_{\bf k} = \Big(\fl{h}{2\pi}\Big)^d \int_{-\pi/h}^{\pi/h}\cdots\int_{-\pi/h}^{\pi/h} |\bxi|^\ap \bigg(\sum_{{\bf j} \in \mathbb{Z}^d} u_{\bf j}\,e^{i\bxi \cdot (\bx_{\bf k}-\bx_{\bf j})}\bigg)\md \bxi, \qquad \mbox{for} \ \ {\bf k} \in {\mathbb Z}^d. \quad 
\eeas
Hence,  the scheme for the  $d$-dimensional  ($d \ge 1$) fractional Laplacian is given by
\bea\label{nDscheme00}
(-\Dt)^{\fl{\ap}{2}}_{h} u_{\bf k}  =  \sum_{{\bf j} \in \mathbb{Z}^d} \upsilon_{\ap, h}^{(d)}(|{\bf k} - {\bf j}|) \, u_{\bf j}, &\quad \mbox{for} \ \, {\bf k} \in {\mathbb Z}^d
\eea
where the coefficients are defined as
\bea\label{coef00}
\upsilon_{\ap, h}^{(d)}(|{\bf k} - {\bf j}|) = \Big(\fl{h}{2\pi}\Big)^d\int_{\bxi\in [-\fl{\pi}{h}, \fl{\pi}{h}]^d} |\bxi|^\ap e^{i\bxi \cdot (\bx_{\bf k}-\bx_{\bf j})} \md \bxi, \quad \ \mbox{for} \ {\bf k}, {\bf j} \in {\mathbb Z}^d.
\eea

For $d = 1$, the scheme (\ref{nDscheme00}) reduces to (\ref{1Dscheme0}) for one-dimensional cases, and we can analytically formulate $\upsilon_{\ap, h}^{(1)}(|{\bf k} - {\bf j}|)$ in terms of generalized hypergeometric functions; see (\ref{1Dcoeff}). 
In the special case of  $\ap = 2$, the coefficients in (\ref{coef00}) can be analytically formulated as: 
\beas
\upsilon_{2, h}^{(d)}(|{\bf k} - {\bf j}|) = \fl{1}{h^{2}}\left\{\begin{array}{lll}
\displaystyle \fl{d}{3}\pi^2, &\ & \mbox{if} \ \ {\bf k} = {\bf j} \\
\displaystyle 2\fl{(-1)^{k_m-j_m}}{(k_m-j_m)^2}, & &\mbox{if} \  k_m \neq j_m,  \  \mbox{and} \ \,  k_l=j_l,  \  \mbox{for} \ \, l \neq m, \\
\displaystyle 0, & &\mbox{otherwise},
 \end{array}\right.\qquad 
\eeas
for any dimension $d \ge 1$. 
For $\ap=2$, the $d$-dimensional semi-discrete Fourier transform degenerates to 1-dimensional semi-discrete Fourier transform along each direction. 
However, for $\ap \in (0, 2)$ and $d > 1$,  it is challenging to obtain the analytical form of $\upsilon_{\ap, h}^{(d)}(|{\bf k} - {\bf j}|)$, and thus numerical integrations are required in practical implementations. 
If  the value of $|{\bf k - \bf j}|$ is large, the integrand in (\ref{coef00}) is a highly oscillatory function.  
Special numerical integration techniques are needed for accurate computation of (\ref{coef00}). 
More discussion on numerical integration of highly oscillating functions can be found in  \cite{Olver2005, Milovanovic2014} and references therein. 

The high-dimensional scheme (\ref{nDscheme00}) is straightforwardly derived following the approach used in one-dimensional cases. 
Consequently, its error analysis  can be conducted similarly to Section \ref{section3}. }

{\subsection{Generalization II}
\label{section4-2}

As previously discussed, numerical integrations are required to compute the coefficients in the high-dimensional scheme (\ref{nDscheme00}). 
To avoid numerically calculating the coefficients, we introduce a new generalization approach in this section.
First, we rewrite the one-dimensional scheme (\ref{1Dscheme0})  as:
\bea\label{1Dscheme00}
(-\Dt)^\fl{\ap}{2} u_k = \sum_{j \in \mathbb Z} \bigg(\fl{h}{2\pi}\int_{-\pi/h}^{\pi/h} |\xi|^\ap e^{i\xi(x_k-x_j)} \md\xi\bigg) u_j, \quad \  \mbox{for} \ \, {k} \in {\mathbb Z},
\eea
which can be viewed as a weighted summation of all points $u_j$,  for $j \in {\mathbb Z}$. 
Denote $B_{r}({\bf 0}) := \big\{\bx \in\mathbb{R}^d : |\bx| \leq r \big\}$ as a $d$-dimensional ball with radius $r > 0$.  
Then we generalize the one-dimensional weighted summation in (\ref{1Dscheme00}) to high dimensions as: 
\bea\label{nDscheme0}
(-\Dt)^{\fl{\ap}{2}}_{h} u_{\bf k} =  \sum_{{\bf j} \in \mathbb{Z}^d} w_{\ap, h}^{(d)}(|{\bf k}-{\bf j}|) u_{\bf j},  \quad \   \mbox{for} \ \, {\bf k} \in {\mathbb Z}^d.
\eea 
with the $d$-dependent weight function 
\bea\label{coef0}
w_{\ap, h}^{(d)}(|{\bf k}-{\bf j}|) = \Big(\fl{h}{2\pi}\Big)^d\int_{\bxi\in B_{\fl{\pi}{h}}({\bf 0})} |\bxi|^\ap e^{i\bxi \cdot (\bx_{\bf k}-\bx_{\bf j})} \md \bxi, \quad \ \mbox{for} \ \, {\bf k}, {\bf j} \in {\mathbb Z}^d.
\eea
by accounting for the rotational invariance of the fractional Laplacian. 
Note that both schemes (\ref{nDscheme00}) and (\ref{nDscheme0}) reduce to the one-dimensional scheme (\ref{1Dscheme0}) if  $d = 1$.

Next we focus on finding the analytical expression of the weight function in (\ref{coef0}). 
}
It is well-known that if $\varphi(\bx) = \varphi(|\bx|)$ is a radial function, then its Fourier transform $\widehat{\varphi}(\bxi)$ is also a radial function.  
Moreover, it can be given in terms of the Hankel transform, i.e., 
\bea\label{Hankel}
 \widehat{\varphi}(\bxi) = \fl{(2\pi)^{d/2}}{|\bxi|^{d/2-1}}\int_{0}^{\infty} \varphi(\rho)\,\rho^{d/2}  J_{d/2-1}(\rho|\bxi|)\md \rho, \quad \ \ \mbox{for} \ \, d \ge 2.
\eea
where $J_{\nu}(z)$ denotes the $\nu$-th order Bessel function of the first kind. 
As an extension, we have the following lemma \cite{Samko1993, Samko2001}: 

\begin{lemma} 
\label{lem-highD} 
If $\varphi(\bx)$ is a summable radial function in a $d$-dimensional ball $B_{r}({\bf 0})$,  then we have
\begin{equation}\label{bessel1}
    \int_{\by\in B_{r}({\bf 0})} \varphi(|\by|) e^{i {\bx\cdot\by}}\md\by = \fl{(2\pi)^{d/2}}{|\bx|^{d/2-1}}\int_{0}^{r} \varphi(\rho)\,\rho^{d/2}  J_{d/2-1}(\rho|\bx|)\md \rho,\qquad
\end{equation}
for any $r > 0$.
\end{lemma}
It is easy to see that if $r \to \infty$, the result in (\ref{bessel1}) collapses to the Fourier transform of  $\varphi(\bx)$ in (\ref{Hankel}).  
The proof of Lemma \ref{lem-highD} can be found in \cite[\S 25.1]{Samko1993}. 
Particularly, if $\varphi(\bx)$ is a radial power function, i.e. $\varphi(\bx) = |\bx|^p$ for $p > 0$, the integral at the right hand side of (\ref{bessel1}) can be calculated exactly, 
i.e., 
\begin{equation}\label{bessel2}
\int_{0}^{r} \rho^\lambda J_{\nu}(\rho) \md\rho = \fl{r^{\lambda+\nu+1}}{2^\nu (\lambda+\nu+1)\Gamma{(\nu+1)}} \,{ }_{1}F_{2}\Big( \fl{\lambda+\nu+1}{2};\nu+1,\fl{\lambda+\nu+3}{2};-\fl{r^2}{4} \Big), 
\end{equation}
for  $r > 0$ and ${\rm Re}{(\lambda+\nu)}>-1$. 
If $r \to \infty$, there is 
\beas
\int_{0}^{\infty} \rho^\lambda J_{\nu}(\rho) \md\rho = 2^{\lambda}\fl{\Gamma\big(({\nu+\lambda+1})/{2}\big)}{\Gamma\big(({\nu-\lambda+1})/{2}\big)}.
\eeas 
Using Lemma \ref{lem-highD} with $r = \pi/h$ and $\varphi(|{\bf y}|) = |{\bf y}|^\ap$, and noticing $|\bx_{\bf k}-\bx_{\bf j}|  = |{\bf k} - {\bf j}| h$,  we obtain 
\beas
&&w_{\ap, h}^{(d)}(|{\bf k}-{\bf j}|) = \Big(\fl{h}{2\pi}\Big)^d \fl{(2\pi)^{d/2}}{|\bx_{\bf k}-\bx_{\bf j}|^{(d-2)/2}}\int_{0}^{\fl{\pi}{h}}\rho^{\ap+d/2} J_{d/2-1}\big(\rho|\bx_{\bf k}-\bx_{\bf j}|\big)\md \rho\qquad\qquad\quad  \nn\\
&&\hspace{2.2cm} = \fl{(2\pi)^{-d/2}}{|{\bf k}-{\bf j}|^{\ap + d} h^{\ap}}\int_{0}^{|{\bf k}-{\bf j}|\pi}  \theta^{\ap+d/2} J_{{d}/{2}-1}(\theta) \md\theta,\qquad
\eeas
where  $\theta = | {\bf k} - {\bf j}|\rho h$. 
We can further use the property of the Bessel functions in (\ref{bessel2}) to obtain 
\bea\label{nDcoeff}
w_{\ap, h}^{(d)}(|{\bf k}-{\bf j}|) 
= \fl{2 \pi^{\ap+3d/2}}{(\ap + d)\Gamma(d/2)h^{\ap+d}}\,_{1}F_{2}\Big( \fl{\ap+d}{2}; \, \fl{\ap+d+2}{2}, \fl{d}{2};\, -\fl{\pi^2|{\bf k} - {\bf j}|^2}{4} \Big),
\eea
 i.e., taking $r = \pi|{\bf k}-{\bf j}|$, $\lambda = \ap + d/2$, and $\nu = d/2-1$ in (\ref{bessel2}). 
 
Substituting (\ref{nDcoeff}) into (\ref{nDscheme0}) immediately yields the spectral approximation to the $d$-dimensional ($d \ge 1$) fractional Laplacian as: 
\bea\label{nDscheme}
&&(-\Dt)^{\fl{\ap}{2}}_{h} u_{\bf k} =  \fl{2^{1-d}\pi^{\ap+d/2}}{(\ap+d)\Gamma{(d/2\big)}h^\ap}\sum_{{\bf j} \in {\mathbb Z}^d} \,_{1}F_{2}\Big( \fl{\ap+d}{2}; \, \fl{\ap+d+2}{2}, \fl{d}{2};\, -\fl{\pi^2|{\bf k} - {\bf j}|^2}{4} \Big) u_{\bf j}, \qquad 
\eea
for ${\bf k} \in {\mathbb Z}^d$ and $\ap > 0$. 
It shows that the structure of our method in (\ref{nDscheme}) is similar to finite difference methods. 
But our method can achieve much higher accuracy. 
Moreover, the coefficients of  finite difference methods are usually given by $d$-dimensional integrals \cite{Duo2019-FDM, Hao2021, Minden2020}, which  require numerical integration to calculate -- the larger the dimension $d$, the higher the computational cost in computing their coefficients. 
In contrast,  dimension $d$ in our method serves as parameters of function $_1F_2$, and thus  the computational cost in calculating $\og^{(d)}_{\ap,\,h}$ is independent of $d$. 

{\subsection{Comparison and discussion}
\label{section4-3}

The main difference between schemes (\ref{nDscheme00}) and (\ref{nDscheme0}) lies in their coefficient functions. 
In (\ref{nDscheme0}), the function  $\og^{(d)}_{\ap,\,h}(|{\bf k}-{\bf j}|)$ can be analytically expressed in terms of generalized hypergeometric functions,  which can be accurately computed with the algorithms in \cite{Pearson2017, Johansson2019}.
In contrast, computing coefficients $\upsilon^{(d)}_{\ap,\,h}(|{\bf k}-{\bf j}|)$ in (\ref{nDscheme00}) requires numerical integration.  
Special numerical integration techniques are particularly demanded if  $|{\bf k}-{\bf j}|$ is large, as the integrand becomes highly oscillating in this case. 
Clearly, the accuracy of scheme (\ref{nDscheme00}) depends on the accuracy of coefficient approximations, and the implementation of (\ref{nDscheme0}) leads to higher computational costs compared to (\ref{nDscheme00}).  

In the following, we conduct numerical experiments to compare the performance of (\ref{nDscheme00}) and (\ref{nDscheme0}) in approximating the two- and three-dimensional fractional Laplacians. 
Note that if $d = 1$, both reduce to the same scheme as in (\ref{1Dscheme0}).

\bb
\noindent{\bf Example 4.1} (Two-dimensional cases).  
Table \ref{Tab4-1} presents the numerical errors in approximating function $(-\Dt)^\fl{\ap}{2}u(x, y)$ on $(-1.5, 1.5)^2$, where  $u(x, y) = e^{-a^2(x^2+y^2)}$ with $a = 6$. 
In this case, the exact solution is given by
\bea\label{ex4-1-1}
(-\Dt)^\fl{\ap}{2} u(x, y) =  (2a)^{\ap} \Gamma\big(\fl{\ap+2}{2}\big) {}_{1}F_{1}\big( \fl{\ap+2}{2};\, 1;\, -a^2(x^2+y^2) \big), \quad\ \mbox{for} \ \, (x, y)\in {\mathbb R}^2.
\eea
The scheme (\ref{nDscheme0}) shows spectral accuracy for any $\ap \in (0, 2]$, while spectral accuracy of (\ref{nDscheme00}) is only observed when $\ap = 2$.  
For large mesh size (e.g. $h = 1/8, 1/16$), both schemes show similar numerical errors. 
However, the errors of scheme (\ref{nDscheme00}) stop decreasing when mesh size reduces to $h = 1/32$ if $\ap \in (0, 2)$, since  they are dominated by the errors in computing coefficients $\upsilon_{d, h}^{(d)}$. 
\begin{table}[htb!]
    \centering
    \begin{tabular}{|c|c|c|c||c|c|c|} \hline
       $\ap$  &$h=1/8$ &$h=1/16$ &$h=1/32$ &$h=1/8$ &$h=1/16$ &$h=1/32$  \\ \hline
       & \multicolumn{3}{|c||}{Generalization I in (\ref{nDscheme00}) with (\ref{coef00}) } & \multicolumn{3}{|c|}{Generalization II in (\ref{nDscheme0}) with (\ref{nDcoeff})} \\ \hline
        0.5     &2.765e-3 &4.278e-7 &6.427e-7 &4.575e-2 &1.416e-7 &2.700e-15\\ \hline
        1.0    &2.823e-2 &2.212e-6 &6.716e-6 &0.251     &1.024e-6 &2.310e-14\\ \hline
        1.7    &0.474     &2.199e-5 &8.014e-5 &2.703     &1.633e-5 &3.837e-13 \\ \hline
        2.0    &1.491     &1.721e-6 &1.089e-12 &7.477     &5.352e-5 &1.677e-12 \\ \hline
    \end{tabular}
    \caption{Numerical errors $\|e^h_{\Dt}\|_{l^{\infty}}(\Og)$ of the two-dimensional schemes  (\ref{nDscheme00}) and (\ref{nDscheme0}) in approximating function $(-\Dt)^{\fl{\ap}{2}}u(x, y)$ on $\Og=(-1.5,1.5)^2$, where $u(x, y)$ is defined in (\ref{ex4-1-1}) with $a = 6$.}
    \label{Tab4-1}
    \end{table}
Here, the coefficients in (\ref{coef00}) for $\ap \in (0, 2)$ are numerically  computed using the MATLAB built-in function {\it integral} with a small absolute error tolerance. 

Table \ref{Tab4-2}  presents the numerical errors in approximating $(-\Dt)^\fl{\ap}{2}u(x, y)$ on a unit disk $B_1({\bf 0})$, where $u(x, y) = \big[1-(x^2+y^2)\big]^4_+$  is a compactly support function. 
   \begin{table}[htb!]
    \centering
    \begin{tabular}{|c|c|c|c|c|} \hline
       $\ap$   &$h=1/16$ &$h=1/32$ &$h=1/64$ &$h=1/128$  \\ \hline
        \multicolumn{5}{|c|}{Generalization I in (\ref{nDscheme00}) with (\ref{coef00})} \\ \hline
        0.5    &5.584e-6 &7.956e-7 &1.468e-6 &2.516e-6\\ \hline
        1.0    &8.608e-5 &1.300e-5 &2.250e-5 &4.915e-5\\ \hline
        1.7    &2.652e-3 &6.035e-4 &4.373e-4 &1.537e-3\\ \hline
        2.0    &1.095e-2 &2.979e-3 &7.708e-4 &1.974e-4\\ \hline
       \multicolumn{5}{|c|}{Generalization II in (\ref{nDscheme0}) with (\ref{nDcoeff})} \\ \hline
        0.5    &3.430e-5  &3.007e-6 &2.643e-7  &2.365e-8\\ \hline
        1.0    &2.692e-4  &3.353e-5 &4.214e-6  &5.289e-7\\ \hline
        1.7    &4.859e-3  &9.852e-4 &2.135e-4  &4.206e-5\\ \hline
        2.0    &1.799e-2  &4.211e-3 &1.152e-3  &2.762e-4\\ \hline
    \end{tabular}
    \caption{Numerical errors $\|e^h_{\Dt}\|_{l^{\infty}(\Og)}$ of the two-dimensional schemes  (\ref{nDscheme00}) and (\ref{nDscheme0}) in approximating function $(-\Dt)^{\fl{\ap}{2}}u(x, y)$ on $B_1({\bf 0})$, where $u(x, y)$ is defined in (\ref{ex4-1-2}).}
    \label{Tab4-2}
    \end{table}
The exact solution is given by 
\bea\label{ex4-1-2}
(-\Dt)^\fl{\ap}{2}u(x, y) = 24\fl{2^\ap \Gamma\big(\fl{\ap+2}{2}\big)}{\Gamma\big(5-\fl{\ap}{2}\big)} {}_{2}F_{1}\big(\fl{\ap+2}{2},-4+\fl{\ap}{2};\,1;\, (x^2+y^2)\big), \quad \mbox{for} \ \ (x, y) \in B_1({\bf 0}). 
\eea
Compared to (\ref{ex4-1-1}), this function has less smoothness at the boundary of the disk. 
In this case, scheme (\ref{nDscheme0}) demonstrates an accuracy of ${\mathcal O}(h^{4-\alpha})$ for any $\alpha \in (0, 2]$. 
The same order of accuracy is observed for scheme (\ref{nDscheme00}) only when $\ap = 2$. 
The accuracy of scheme (\ref{nDscheme00}) is deteriorated by the errors of numerical integration if small mesh size is used (e.g., $h = 1/64, 1/128$). 
The performance of the scheme  (\ref{nDscheme00}) can be improved by using more accurate numerical integration (e.g. in \cite{Olver2005, Milovanovic2014}) for highly oscillating functions, which will be explored in our  future study.

\bigskip
\noindent{\bf Example 4.2} (Three-dimensional cases).  
Table \ref{Tab4-3} further compares these two schemes in approximating the three-dimensional function $(-\Dt)^\fl{\ap}{2}u(x, y, z)$ on $(-1, 1)^3$, where  $u(x, y, z) = xye^{-a^2|\bx|^2}$ with $a = 7$. 
In this case, the exact solution is given by 
\bea\label{ex4-2}
(-\Dt)^\fl{\ap}{2} u(x, y, z) = \fl{2^{3+\ap} a^{\ap}\Gamma\big(\fl{\ap+7}{2}\big)}{15\sqrt{\pi}} xy \,{}_{1}F_{1}\big(\fl{\ap+7}{2};\, \fl{7}{2};\,-a^2(x^2+y^2+z^2)\big), \,\, \
\eea
for $(x, y, z) \in {\mathbb R}^3$. 
The observations are consistent with  those from Table \ref{Tab4-1} for two-dimensional cases. 
The scheme (\ref{nDscheme0}) shows spectral accuracy for any $\alpha \in (0, 2]$, while the accuracy of scheme (\ref{nDscheme00}) is significantly influenced by the accuracy of numerical integration when $\alpha \in (0, 2)$. 
\begin{table}[ht!]
    \centering
    \begin{tabular}{|c|c|c|c||c|c|c|c|} \hline
       $\ap$  & $h=1/4$ &$h=1/8$ &$h=1/16$  & $h=1/4$ &$h=1/8$ &$h=1/16$ &$h=1/32$\\ \hline
      & \multicolumn{3}{|c||}{Generalization I in (\ref{nDscheme00}) with (\ref{coef00})} &\multicolumn{4}{|c|}{Generalization II in (\ref{nDscheme0}) with (\ref{nDcoeff})}\\ \hline
        0.5   &4.784e-4   &1.062e-3 &2.597e-3 &2.900e-4  &2.239e-3 &2.138e-6 &1.100e-17  \\ \hline
        1.0   &3.501e-3   &9.083e-3 &3.096e-2 &2.768e-3  &1.120e-2 &1.542e-5 &1.390e-16 \\ \hline
        1.7   &3.596e-2   &0.123    &0.666    &3.102e-2  &0.106    &2.454e-4 &4.219e-15 \\ \hline
        2.0   &8.933e-2   &6.838e-2 &4.179e-6 &7.920e-2  &0.276    &8.034e-4 &1.177e-14 \\ \hline
    \end{tabular}
    \caption{Numerical errors $\|e^h_{\Dt}\|_{l^{\infty}(\Og)}$ of the three-dimensional schemes  (\ref{nDscheme00}) and (\ref{nDscheme0}) in approximating function $(-\Dt)^{\fl{\ap}{2}}u(x, y, z)$ on $\Og = (-1, 1)^3$, where $u(x, y, z)$ is defined in (\ref{ex4-2}). Note that errors in scheme (\ref{nDscheme00}) stop decreasing for $h = 1/16$ due to integration errors when $\alpha < 2$. However, for $\alpha = 2$, scheme (\ref{nDscheme00}) can achieve an error of 1.266e-14 for $h=1/32$. }
    \label{Tab4-3}
    \end{table}
Moreover, compared to two-dimensional cases, the numerical integration required for evaluating (\ref{coef00}) in three-dimensional cases takes significantly longer time.
This again highlights the advantages of scheme (\ref{nDscheme0}), as the coefficients can be calculated more efficiently and accurately by expressing them in terms of the generalized hypergeometric function. 
}

If a bounded domain with extended Dirichlet boundary conditions is considered, {both scheme in (\ref{nDscheme00}) and (\ref{nDscheme0})} can be written into matrix-vector form with a (multi-level) Toeplitz matrix.  
Then the approximation of the  fractional Laplacian can be efficiently evaluated via the $d$-dimensional fast Fourier transforms with computational cost ${\mathcal O}(2N\log (2N))$ and memory cost ${\mathcal O}(N)$, where $N$ denotes the total number of grid points in domain.

\section{Numerical results}
\label{section5}
\setcounter{equation}{0}

In this section, we examine the performance of our method and verify our theoretical results. 
To this end,  we first test the numerical accuracy of our method in approximating the classical and fractional Laplacian $(-\Dt)^{\fl{\alpha}{2}}$ (i.e. for $0 < \ap \le 2$), and then apply it to study fractional elliptic problems. 
As discussed previously, our method has the similar framework as finite difference methods, but it can achieve much higher accuracy.  
To demonstrate this, we will compare our method with some representative finite difference methods in the literature \cite{Huang2014, Duo2018, Hao2021, Minden2020}. 
{Unless otherwise stated, we will always use the scheme (\ref{nDscheme0}) with (\ref{nDcoeff}) in the following simulations.}

\subsection{Discretization of the fractional Laplacian}
\label{section5-1}

In the following, we test the accuracy of our method in approximating $(-\Dt)^\fl{\ap}{2}u$ under different  smoothness conditions of $u$. 
The error function is defined as
\beas
\big(e_\Dt^h\big)_{\bf j}  = (-\Dt)^{\fl{\ap}{2}}u(\bx_{\bf j}) - (-\Dt)^{\fl{\ap}{2}}_{h,d} u(\bx_{\bf j}), \quad \ \ \mbox{for} \ \, \bx_{\bf j} \in \Og,
\eeas
where $(-\Dt)_{h, d}^\fl{\ap}{2}$ is the numerical approximation of the $d$-dimensional fractional Laplacian in (\ref{nDscheme}). 

\begin{example}\label{ex-op1}
Consider an inverse multiquadratic function of the form:
\begin{equation}\label{ex1}
u(x) = \fl{1}{(1+x^2)^7}, \quad \  \mbox{for} \ \, x\in\mathbb{R}.
\end{equation}
The exact solution of $(-\Dt)^{\fl{\alpha}{2}}u(x)$ is given by \cite{Wu2022},
\begin{equation}\label{exact1}
    (-\Dt)^{\fl{\alpha}{2}}u(x) = \fl{2^\ap \Gamma{\big(\fl{1+\ap}{2}\big)} \Gamma{\big(7+\fl{\ap}{2}\big)}}{720\sqrt{\pi}} \, {}_{2} F_{1}\bigg(\fl{\ap+1}{2}, 7+\fl{\ap}{2}; \fl{1}{2}; -x^2\bigg), \quad \ \mbox{for} \ \, x\in\mathbb{R}, 
\end{equation}
for any $\ap > 0$.  We will numerically approximate the function $(-\Dt)^{\fl{\alpha}{2}}u(x)$ on $\Og =  (-1, 1)$. 
\end{example}

Table \ref{tab:ex1-1} presents numerical errors $\|e_\Dt^h\|_{l^\infty(\Og)}$ of our method  for different power $\ap$ and mesh size $h$.  
Here, the function $u \in C^\infty({\mathbb R})$. 
Table \ref{tab:ex1-1} shows that our method has a spectral accuracy for any $\ap > 0$, confirming our analytical results in Theorem \ref{thm1} (ii). 
Moreover, for the same mesh size $h$, the larger the power $\ap$, the bigger the numerical errors. 
\begin{table}[htb!]
    \centering
    \begin{tabular}{ccccc}\hline
    &$h = 1/2$ & $h=1/4$ & $h=1/8$ & $h=1/16$ \\
        \hline
       $\ap =  0.5$    & 0.068 & 9.007e-4 &4.641e-8  &1.723e-16 \\
       $\ap = 1.0$     & 0.325 & 6.319e-3 &4.645e-7  &2.112e-16  \\
       $\ap = 1.7$     & 2.003 & 6.301e-2 &7.540e-6  &5.204e-15  \\
        $\ap = 2.0$    & 4.165 & 1.588e-1 &2.335e-5  &1.979e-14   \\
    \hline
    \end{tabular}
\caption{Numerical errors $\|e_{\Dt}^h\|_{l^\infty(\Og)}$ of our method in approximating $(-\Dt)^{\fl{\ap}{2}}u$ on $\Og = (-1, 1)$, where $u$ is defined in (\ref{ex1}).}
\label{tab:ex1-1}
\end{table}

Next, we compare our method with some representative finite difference methods in the literature \cite{Huang2014, Duo2018, Hao2021, Minden2020}.  
Table \ref{tab:ex1-2} shows their numerical errors with much smaller mesh size, e.g., $h = 1/256, 1/512$. 
From Tables \ref{tab:ex1-1} and \ref{tab:ex1-2}, we find that our method can achieve the same accuracy with much less number of points $N$. 
For instance, to achieve errors of ${\mathcal O}(10^{-6})\sim{\mathcal O}(10^{-8})$, our method needs around $N = 16$ (i.e., $h = 1/8$), but the methods in Table \ref{tab:ex1-2} requires $N = 1024$ (i.e., $h = 1/512$) or more. 
Consequently,  the differentiation matrices of these methods are much larger than that of our method. 
\begin{table}[ht!]
\centering
\begin{tabular}{cccccc}\hline
methods & & Huang \cite{Huang2014} & Duo \cite{Duo2018} & Hao \cite{Hao2021} & Minden  \cite{Minden2020}\\
\hline
\multirow{ 2}{*}{$\ap = 0.5$}  & $h = 1/256$  & 1.814e-4 &5.521e-6 &1.085e-5 & 4.047e-6 \\
                               & $h = 1/512$  & 6.864e-5 &1.383e-6 &2.712e-6 & 5.710e-6\\
\hline 
\multirow{ 2}{*}{$\ap = 1.0$}  & $h = 1/256$  & 2.711e-3 &7.336e-7 &5.594e-5 &5.314e-7 \\
                               & $h = 1/512$  & 1.381e-3 &1.170e-7 &1.398e-5 &6.673e-8 \\
\hline 
\multirow{ 2}{*}{$\ap = 1.7$}  & $h = 1/256$  & 5.423e-2 &1.794e-4 &3.878e-4 &9.800e-5 \\
                               & $h = 1/512$  & 4.430e-2 &4.154e-5 &9.695e-5 &1.940e-5 \\
\hline 
\end{tabular}
\caption{Numerical errors $\|e_\Dt^h\|_{l^\infty(\Og)}$ of finite difference methods in \cite{Huang2014, Duo2018, Hao2021, Minden2020} for approximating  $(-\Dt)^{\fl{\ap}{2}}u$ on $\Og = (-1, 1)$, where $u$ is defined in (\ref{ex1}).}
\label{tab:ex1-2}
\end{table} 
Note that the computational cost in computing matrix-vector product is ${\mathcal O}(2N\log(2N))$. 
Hence,  our method has much lower computational and storage cost. 
It is known that the methods in \cite{Duo2018, Hao2021} have accuracy of ${\mathcal O}(h^{2})$, while the method in \cite{Huang2014} has ${\mathcal O}(h^{2-\ap})$, which is confirmed by our observations in Table \ref{tab:ex1-2}. 

\bb
\begin{example}\label{ex-op2}
We consider a compact support function
\begin{equation}\label{ex2}
    u(x) = (a^2 - x^2)^s_{+}, \qquad \text{ for} \ \, x \in {\mathbb R}.
\end{equation}
It is easy to see that $u\in C^{s-1,1}(\mathr)$ for $s\in\mathbb{N}$.  
For $s \notin {\mathbb N}$,  it satisfies $u \in C^{\lfloor s\rfloor,\,s-\lfloor s\rfloor}({\mathbb R})$ with $\lfloor \cdot \rfloor$ denoting the floor function. 
In this case, the exact solution of $(-\Dt)^{\fl{\alpha}{2}}u$ is given by \cite{Dyda2012, Wu2022}:
\begin{equation}
    (-\Dt)^{\fl{\alpha}{2}} u(x) = \fl{2^\ap \Gamma{\big(\fl{\ap+1}{2}\big)}\Gamma{(s+1)}a^{2s-\ap}}{\sqrt{\pi}\,\Gamma{(s+1-\fl{\ap}{2})}} \,{ }_2 F_{1}\bigg(\fl{\ap+1}{2}, -s+\fl{\ap}{2};\,  \fl{1}{2}; \, \fl{x^2}{a^2}\bigg), 
\end{equation}
for any $\ap \ge 0$ and $|x|<a$. 
\end{example}

Take domain $\Og = (-1, 1)$. 
We study numerical errors in approximating function $(-\Dt)^{\fl{\ap}{2}}u(x)$ on $\Og$.  
 If choosing $a = 1$, the function in (\ref{ex2}) also satisfies $u \in C^{\lfloor s\rfloor,\,s-\lfloor s\rfloor}(\bar{\Og})$, and it has less smoothness at points $x = \pm 1$. 
\begin{table}[ht]
    \centering
    \begin{tabular}{cllllll}\hline
 &$h=1/8$ &$h=1/16$ &$h=1/32$ &$h=1/64$ &$h=1/128$ &$h=1/256$ \\ \hline
     \multirow{2}{*}{$\ap = 0.5$} &5.611e-5 &4.963e-6 &4.351e-7 &3.823e-8 &3.367e-9 &2.970e-10 \\
            &c.r. &3.499 &3.512 &3.509 &3.505 &3.503 \\
            \cline{2-7}
      \multirow{2}{*}{$\ap = 1.0$} &5.448e-4 &6.751e-5  &8.330e-6 &1.032e-6 &1.284e-7 &1.601e-8 \\
            &c.r. &3.013&3.019 &3.012 &3.007 &3.004 \\ 
                        \cline{2-7}
      \multirow{2}{*}{$\ap = 1.7$} &8.537e-3 &1.713e-3 &3.428e-4 &6.895e-5 &1.393e-5 &2.820e-6 \\
            &c.r. &2.3171 &2.321 &2.314 &2.308 &2.304 \\
                        \cline{2-7}
     \multirow{2}{*}{$\ap = 2.0$} &2.610e-2 &6.454e-3 &1.590e-3 &3.939e-4 &9.796e-5 &2.442e-5 \\ 
            &c.r. &2.016 &2.021 &2.014 &2.008 &2.004 \\ \hline
    \end{tabular}
    \caption{Numerical errors $\|e_\Dt^h\|_{l^\infty(\Og)}$ and convergence rate (c.r.) of our method in approximating  $(-\Dt)^{\fl{\ap}{2}}u(x)$ on $\Og = (-1, 1)$, where $u$ is defined in \eqref{ex2} with $a=1$ and $s=4$.}
    \label{tab1:ex2}
\end{table}
Tables \ref{tab1:ex2} presents numerical errors and convergence rates of our method, where $u$ is from (\ref{ex2}) with $a=1$ and $s=4$. 
It is easy to verify that $u \in C^{3, 1}({\mathbb R})$, and $(-\Dt)^\fl{\ap}{2}u \in C^{3-\lfloor\ap\rfloor,\,1-\ap+\lfloor\ap\rfloor}({\mathbb R})$. 
Hence, the larger the power $\ap$, the less smooth the function $(-\Dt)^\fl{\ap}{2} u$, and the larger the approximation errors. 
We find that our method has an accuracy of $\mathcal{O}(h^{4-\ap})$ for $\|e_\Dt^h\|_{l^\infty(\Og)}$, confirming our analytical results in Theorem \ref{thm1} (i). 
Moreover, our extensive studies show that:  (i) the maximum errors occur around the boundary of $\Og$, i.e., $x = \pm 1$; (ii) numerical error $\|e_\Dt^h\|_{l^2(\Og)}$ is much smaller than $\|e_\Dt^h\|_{l^\infty(\Og)}$, and it has an accuracy of  $\mathcal{O}(h^{4.5-\ap})$. 

Figure \ref{fig:ex2} further compares our method with those representative finite difference methods in \cite{Huang2014, Duo2018, Minden2020, Hao2021}, where $u$ is chosen from (\ref{ex2}) with $a = 1$, and $s = 2+\ap$,  or $s = 8$. 
Generally, the larger the value of $s$, the smoother the function $(-\Dt)^\fl{\ap}{2}u$, and the smaller the numerical errors. 
Hence, numerical errors for $s = 8$ are smaller than those of $s = 2 +\ap$. 
For $s = 2+\ap$, the function $u \in C^{2+\lfloor\ap\rfloor,\,\ap-\lfloor\ap\rfloor}({\mathbb R})$. 
Our method and those in \cite{Duo2018, Hao2021, Minden2020} all have the second order of accuracy, but the errors of our method are much smaller; see Figure \ref{fig:ex2} a) $\&$ b). 
While the method in \cite{Huang2014} has the accuracy of ${\mathcal O}(h^{2-\ap})$, which is not improved even for a smoother function (e.g., $s = 8$). 
\begin{figure}[ht]
\centerline{
a)\includegraphics[width=5.6cm, height = 4.86cm]{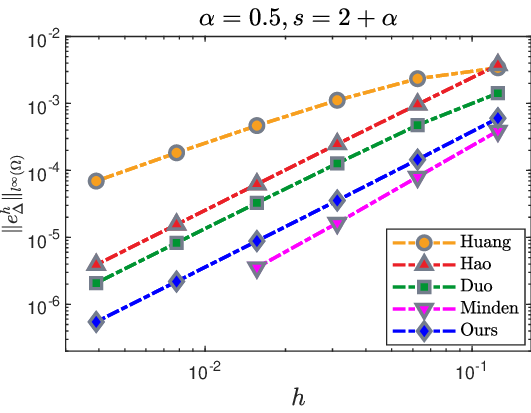}\hspace{-1mm}
b)\includegraphics[width=5.6cm, height = 4.86cm]{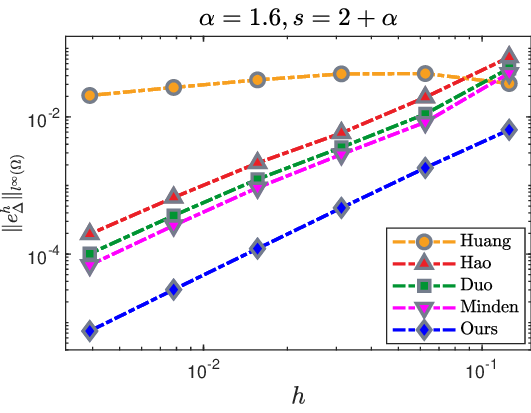}}
\centerline{
c)\includegraphics[width=5.6cm, height = 4.86cm]{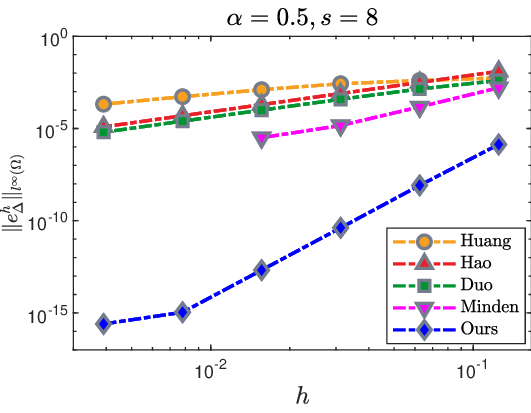}\hspace{-1mm}
d)\includegraphics[width=5.6cm, height = 4.86cm]{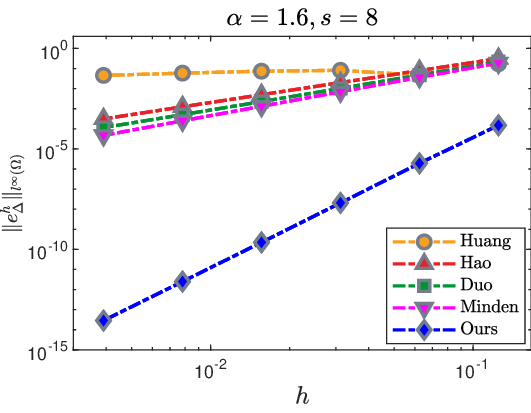}}
\caption{Comparison of our method and others (including {\it Huang} in \cite{Huang2014},  {\it Duo} in \cite{Duo2018},  {\it Minden} in \cite{Minden2020}, and {\it Hao} in \cite{Hao2021})  in approximating $(-\Dt)^{\fl{\ap}{2}}u$ on $\Og = (-1, 1)$, where $u$ is defined in (\ref{ex2}) with $a = 1$. }
\label{fig:ex2}
\end{figure}
Note that the accuracy of finite difference methods in \cite{Duo2018, Hao2021} and \cite{Minden2020} are capped by ${\mathcal O}(h^2)$ and ${\mathcal O}(h^{4-\ap})$, respectively. 
In contrast, the accuracy of our method,  depending on function regularity, is ${\mathcal O}(h^{s-\ap})$; see Figure \ref{fig:ex2} c) \& d) for $s = 8$. 
This is one main advantage of our method in comparison to those in \cite{Duo2018, Huang2014, Minden2020, Hao2021}.

{
Moreover, Table \ref{Tab-Ex2-new} compares their computational times. 
We usually formulate the task of approximating the function $(-\Delta)^{\frac{\alpha}{2}}u$ into a matrix-vector form as $(-\Delta)_h^{\frac{\alpha}{2}}\mathbf{u} = A\mathbf{u}$. 
Hence, the computational time of approximating $(-\Dt)^\fl{\ap}{2}u$ comes from two parts: assembling the matrix $A$ and performing the matrix-vector multiplication $A{\bf u}$.  
Table \ref{Tab-Ex2-new} shows that to obtain an accuracy of ${\mathcal O}(10^{-8})$, our method requires  $N = 64$ points, while the methods in \cite{Huang2014, Duo2018, Hao2021} require $N = 16384$ or more. 
Consequently, their computational time on computing $A{\bf u}$ is significantly longer than ours.  
To compare their time in assembling matrix $A$, we note that the discretizations from our method or those in \cite{Huang2014, Duo2018, Hao2021} all lead to a symmetric (multilevel) Toeplitz matrix $A$.  
Hence, we only need to compute the entries of a single column (or row).  
\begin{table}[ht!]
    \centering
    \begin{tabular}{|c|c|c|c|c|} \hline
    Method & Huang \cite{Huang2014}  & Duo \cite{Duo2018} & Hao \cite{Hao2021} & Ours\\
    \hline
    Error $\|e_\Dt^h\|_{l^\infty(\Og)}$ & 1.148e-2 & 9.009e-8 & 3.044e-7  & 2.055e-8\\
    \hline
    Number of points $N$ & \multicolumn{3}{|c|}{$16384$} & $64$\\
    \hline
   Time in assembling $A$ & 1.150e-3 &6.386e-4 & 2.477e-4  & 7.729e-2 \\
    \hline
    Time in computing $A{\bf u}$ &\multicolumn{3}{|c|}{1.26} & 4.710e-5\\
    \hline
\end{tabular}
\caption{Comparison of the accuracy and computational time (in second) of our method and those in \cite{Huang2014, Duo2018, Hao2021} when approximating $(-\Dt)^{\fl{\ap}{2}}u$ on $\Og = (-1, 1)$, where $\ap = 1.6$ and $u$ is defined in (\ref{ex2}) with $a = 1$ and $s = 8$.}\label{Tab-Ex2-new} 
\end{table}
For methods in \cite{Huang2014, Duo2018, Hao2021}, each entry of matrix $A$ is typically defined by a $d$-dimensional integral.  
In the special case of $d = 1$, it can be analytically integrated and expressed in terms of elementary functions, and thus the computational time required for computing entries is insignificant, as demonstrated in Table \ref{Tab-Ex2-new}. 
However, in  high-dimensional ($d \ge 2$) cases, each entry needs to be computed through numerical integration. 
Consequently, the computational time required to assemble matrix $A$ significantly increases as the dimensions $d$ increases.
In contrast, the  time of computing each entry of matrix $A$ in our method is independent of the dimension $d$. 

Hence, to achieve the same accuracy, our method requires a significantly smaller number of points, resulting in significantly shorter computational time. 
This advantage of our method becomes even more pronounced in higher dimensions.
}
\begin{example}
We consider a two-dimensional compact support function of the form
\begin{equation}\label{ex3}
    u(x,y) = \big[(1-x^2)(1-y^2)\big]^s_{+}, \qquad\mbox{for} \ \ (x,y) \in {\mathbb R}^2, 
\end{equation}
for $s \in {\mathbb N}$. 
In this case, the analytical results of $(-\Dt)^\fl{\ap}{2}u$ is unknown if $\ap < 2$. 
\end{example}

Figure \ref{fig:ex3} illustrates the numerical results of $(-\Dt)^{\fl{\ap}{2}}u$ on domain $\Og = (-1, 1)^2$ for $\ap = 1$, and $s = 2$ or $4$ in (\ref{ex3}). 
It shows that even though $u(\bx) \equiv 0$ for $\bx \in {\mathbb R}^2\backslash\Og$, function $(-\Dt)^{\fl{\ap}{2}}u(\bx)$ is not necessarily zero on  ${\mathbb R}^2\backslash\Og$ due to the nonlocality of the fractional Laplacian. 
\begin{figure}[ht!]
\centerline{
a) \includegraphics[width=4.46cm, height = 3.6cm]{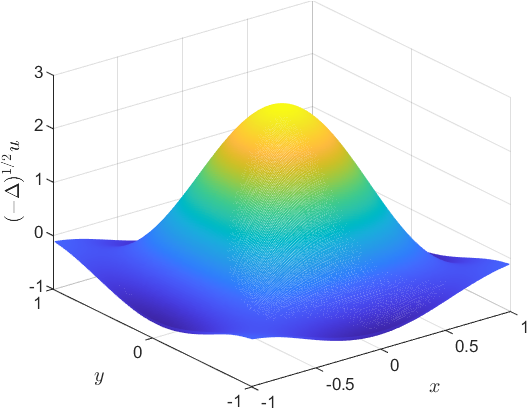}\hspace{5mm}
b) \includegraphics[width=4.46cm, height = 3.6cm]{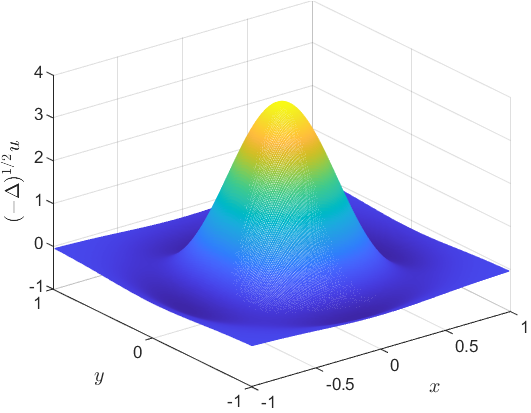}}
\caption{Numerical results of $(-\Dt)^\fl{1}{2} u(x,y)$, where $u$ is from \eqref{ex3} with $s = 2$ (a) or $s = 4$ (b). }
\label{fig:ex3}
\end{figure}
Tables \ref{tab1:ex3}--\ref{tab2:ex3} further show numerical errors of our method for different power $\ap$ and mesh size $h$. 
Since the exact solution of $(-\Dt)^{\fl{\ap}{2}}u$ is unknown, we use the numerical solution with a fine mesh size $h=2^{-10}$ as the reference to compute numerical errors. 
The function in (\ref{ex3}) satisfies $u \in C^{s-1,1}({\mathbb R}^2)$. 
It shows that our method has the accuracy of ${\mathcal O}(h^{s-\ap})$ for $\|e_\Dt^h\|_{l^\infty(\Og)}$, while ${\mathcal O}(h^{s-\ap+1/2})$ for $\|e_\Dt^h\|_{l^2(\Og)}$; see Table \ref{tab1:ex3}. 
\begin{table}[ht!]
    \centering
    \begin{tabular}{llllllll}\hline
& &$h=1/8$ &$h=1/16$ &$h=1/32$ &$h=1/64$ &$h=1/128$ \\ 
\hline
\multirow{4}{*}{$\ap = 0.5$} & \multirow{2}{*}{$\|e_\Dt^h\|_{l^\infty(\Og)}$} &1.274e-3 &4.432e-4 &1.573e-4 &5.581e-5 &1.977e-5 \\
&&c.r. &1.523 &1.494 &1.495 &1.497 \\
  \cline{2-7}
& \multirow{2}{*}{$\|e_\Dt^h\|_{l^2(\Og)}$} & 8.129e-4& 2.028e-4& 5.104e-5 & 1.281e-5& 3.210e-6\\
&&c.r. & 2.003&1.991 &1.994  &1.997 \\
\hline
\multirow{4}{*}{$\ap = 1.0$} & \multirow{2}{*}{$\|e_\Dt^h\|_{l^\infty(\Og)}$}&1.165e-2 &5.791e-3 &2.904e-3 &1.456e-3 &7.291e-4 \\
  &   &c.r. &1.008 &0.996 &0.997 &0.997 \\
   \cline{2-7}
 & \multirow{2}{*}{$\|e_\Dt^h\|_{l^2(\Og)}$} & 7.540e-3 & 2.675e-3 & 9.507e-4 & 3.371e-4 & 1.194e-4\\
   &   &c.r. & 1.495 & 1.492& 1.496 & 1.497\\
 \hline    
\multirow{4}{*}{$\ap = 1.7$} & \multirow{2}{*}{$\|e_\Dt^h\|_{l^\infty(\Og)}$} &1.712e-1 &1.387e-1 &1.129e-1 &9.189e-2 &7.495e-2  \\
 &           &c.r. &0.304 &0.297 &0.297 &0.294 \\
    \cline{2-7}
  & \multirow{2}{*}{$\|e_\Dt^h\|_{l^2(\Og)}$} & 1.121e-1 & 6.470e-2 & 3.734e-2 & 2.151e-2 & 1.240e-2\\
   &           &c.r. &0.793 & 0.793 & 0.796 &0.794\\
   \hline     
\multirow{2}{*}{$\ap = 2.0$} & \multirow{2}{*}{$\|e_\Dt^h\|_{l^2(\Og)}$}  & 0.334 & 0.238 & 0.169& 0.120 & 8.531e-2      \\      
& &c.r. &0.492&0.493 &0.495 &0.490 \\ 
            \hline
    \end{tabular}
    \caption{Numerical errors $\|e_\Dt^h\|_{l^\infty(\Og)}$ and $\|e_\Dt^h\|_{l^2(\Og)}$ and convergence rate (c.r.) of our method in approximating  $(-\Dt)^\fl{\ap}{2}u$ on $\Og = (-1,1)^2$, where $u$ is defined in \eqref{ex3} with $s=2$. Note that for $\ap = 2$ the method does not converge in this case. }
    \label{tab1:ex3}
\end{table}
\begin{table}[ht!]
    \centering
    \begin{tabular}{llllll}\hline
 &$h=1/8$ &$h=1/16$ &$h=1/32$ &$h=1/64$ &$h=1/128$ \\ 
\hline           
 \multirow{2}{*}{$\ap =  0.5$} &7.497e-5 &5.143e-6 &4.368e-7 &3.824e-8 &3.367e-9 \\
      &     c.r. &3.866 &3.558 &3.514 &3.505 \\
\hline
\multirow{2}{*}{$\ap = 1.0$}   &6.383e-4 &6.875e-5 &8.3453-6 &1.033e-6 &1.284e-7 \\
            &c.r. &3.215 &3.042 &3.015 &3.007 \\
\hline
\multirow{2}{*}{$\ap = 1.7$}   &9.453e-3 &1.734e-3 &3.432e-4 &6.896e-5 &1.393e-5  \\
            &c.r. &2.447 &2.336 &2.315 &2.307 \\   
\hline
\multirow{2}{*}{$\ap = 2.0$}   &2.856e-2 &6.523e-3 &1.592e-3 &3.940e-4 &9.799e-5 \\
            &c.r. &2.131 &2.034 &2.015 &2.007 \\ \hline
    \end{tabular}
    \caption{Numerical errors $\|e_\Dt^h\|_{l^\infty(\Og)}$  and convergence rate (c.r.) of our method in approximating  $(-\Dt)^\fl{\ap}{2}u$ on $\Og = (-1,1)^2$, where $u$ is defined in \eqref{ex3} with $s=4$.}
    \label{tab2:ex3}
\end{table}
Note that our method becomes more advantageous in higher dimensions, as fewer points are required to obtain the desired accuracy. 
Hence, it has significantly less computational and storage costs. 

\subsection{Fractional elliptic problems}
\label{section5-2}

In the following, we apply our numerical method to solve the fractional elliptic problems with extended Dirichlet boundary conditions.  
Denote the grid error function
\beas
\big(e_u^h\big)_{\bf j} = u(\bx_{\bf j}) - u^h_{\bf j}, \qquad\mbox{for} \ \,\bx_{\bf j} \in \Og,
\eeas
with $u(\bx_{\bf j})$ and $u_{\bf j}^h$ representing the exact and numerical solutions at point $\bx_{\bf j}$, respectively.

\begin{example}\label{example-poisson1}
Consider the one-dimensional Poisson equation with extended homogeneous Dirichlet boundary conditions: 
\bea\label{elliptic-ex1}
\begin{aligned}
(-\Dt)^\fl{\ap}{2} u(x) = x\,_2F_1\Big(\fl{3+\ap}{2}, -s+\fl{\ap}{2};\, \fl{3}{2}; \, x^2\Big), \qquad &\mbox{for} \ \, x \in (-1, 1),\\
u(x) = 0, \qquad\, &\mbox{for} \ \,  x \in {\mathbb R}\backslash(-1,1), 
\end{aligned}
\eea
where $s > 0$. 
The exact solution of (\ref{elliptic-ex1}) is given by \cite{Dyda2012, Wu2022}:
\beas
u(x) = \fl{\sqrt{\pi}\, \Gamma\big(s+1-\fl{\ap}{2}\big)}{2^{\ap+1}\Gamma\big(s+1\big)\Gamma\big((3+\ap)/2\big)}\, x\big(1-x^2\big)^s_+, 
\eeas
for any $\ap \in (0, 2]$. 
\end{example}

Figure \ref{fig:ex4} presents the  numerical errors and convergence rates of our method for various $\ap$ and $s$, where order lines are included for better comparison.  
Here,   the solution $u \in C^{s-1, 1}(\bar{\Og})$ for $s \in {\mathbb N}$, while $C^{\lfloor s\rfloor,\, s-\lfloor s\rfloor}(\bar{\Og})$ for $s \notin {\mathbb N}$.  It shows that the smoother the solution $u$, the smaller the numerical errors, and the higher the accuracy rate of our method. 
\begin{figure}[ht!]
\centerline{
\includegraphics[width=5.6cm, height = 4.86cm]{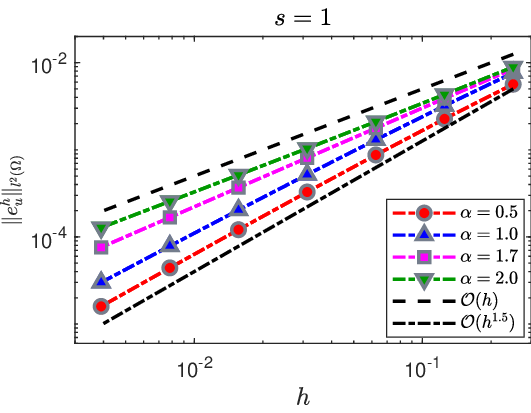} \hspace{3mm}
\includegraphics[width=5.6cm, height = 4.86cm]{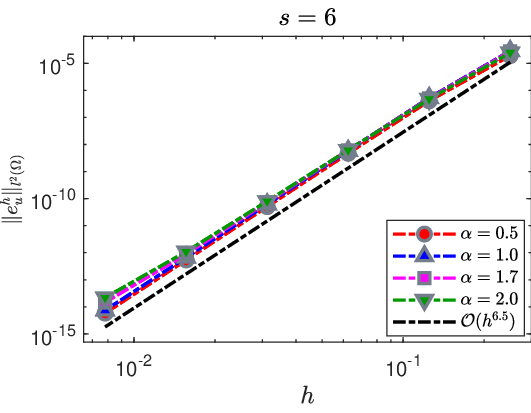}}
\caption{Numerical errors  of our method in solving the Poisson problem (\ref{elliptic-ex1}).}
\label{fig:ex4}
\end{figure}
Numerical errors $\|e_u^h\|_{l^2(\Og)}$ decays with a rate between $s$ and $s+1/2$, confirming our theoretical results in Theorem \ref{thm3} (i). 
More precisely, the observed  accuracy rate is higher than the predicted rate (i.e., ${\mathcal O}(h^{s-\ap+\fl{1}{2}})$). 
In addition, we find that the convergence rate of $\|e_u^h\|_{l^\infty(\Og)}$ is ${\mathcal O}(h^{s})$. 
Our extensive studies show that our method outperforms those in \cite{Huang2014, Duo2018, Minden2020, Hao2021}; we omit showing their details for brevity.

{
\begin{remark}
To further demonstrate the effectiveness of our method, we also measure the errors at non-grid points.  
Let  $u_j^h$ represent the numerical approximation of $u(x_j)$,  obtained by numerically solving the fractional Poisson equation in (\ref{elliptic-ex1}) at grid points $\{x_j\}_{j\in\Og_h}$. 
Denote $\tilde{x}_l = lh$ for $l \in \mathbb{R}$ (instead of ${\mathbb Z}$) to distinguish from the grid points $x_j$. 
Setting $\ap = 0$ in (\ref{1Dscheme}), we obtain the approximate solution of (\ref{elliptic-ex1}) at any point $\tilde{x}_l \in \Omega$:
\bea\label{utilde}
&&\tilde{u}_l^h = \sum_{j\in\Og_h} \,_0F_1\Big(; \fl{3}{2}; \fl{-\pi^2(l-j)^2}{4}\Big)u_j^h\nn\\
&&\hspace{5mm} = \sum_{j\in\Og_h} {\rm sinc}\Big(\fl{\pi}{h}(x_l-x_j)\Big) u_j^h, \qquad\mbox{for} \ \, \tilde{x}_l\in \Og,
\eea
where `$\sim$' is included in $\tilde{u}_l^h$ to indicate that it is computed from (\ref{utilde}). 
Figure \ref{Fig-remark} (a) shows the numerical errors at both grid points $\{x_j\}$ and non-grid points $\{\tilde{x}_l\}$, where we set  $\ap = 1$ and $s = 6$ in (\ref{elliptic-ex1}). 
\begin{figure}[ht!]
\centerline{
 (a)\includegraphics[width=5.86cm, height = 4.86cm]{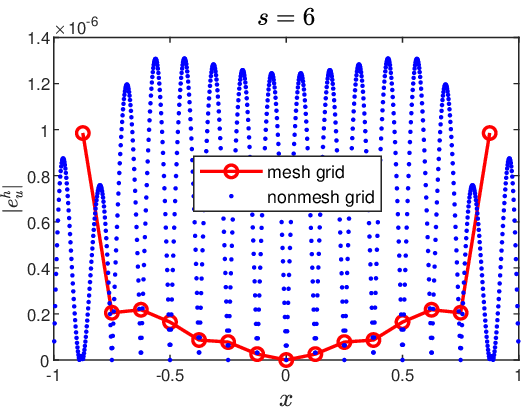}\hspace{3mm}
 (b)\includegraphics[width=5.86cm, height = 4.86cm]{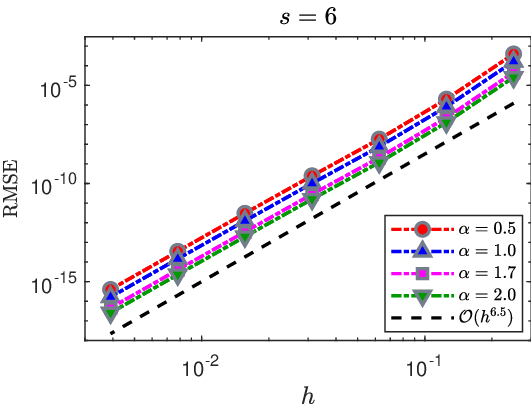}}
 \caption{Numerical errors of our method in solving the Poisson problem (\ref{elliptic-ex1}). (a) Pointwise errors at grid points $\{x_j\}$ (red `$\circ$') and non-grid points $\{\tilde{x}_l\}$ (blue `$\bullet$'), where $\ap = 1$; (b) RMS errors for different $\ap$, where $h$ indicates the mesh size used to compute $u_j^h$ in (\ref{utilde}).}
\label{Fig-remark}
\end{figure}
The points $\{\tilde{x}_l\}_{l \in \tilde{\Og}_h}$ are uniformly distributed in the interval $(-1, 1)$ with a total of $M = 1000$ points, where the index set \,  $\tilde{\Og}_h = \{ l = 2m/125 \mid\,  m \in {\mathbb Z},\ \mbox{and }\, lh \in \Og\}$. 
It shows that our method has small errors not only at grid points but also at non-grid points. 
Moreover, we present the root mean square errors (RMS) in Figure \ref{Fig-remark} (b), in comparison to the $l_2$-norm errors  in Figure \ref{fig:ex4}.
Here, the RMS error is defined as
\beas
\|e_u^h\|_{\rm rms} = \bigg(\fl{1}{M} \sum_{l \in \tilde{\Og}_h} \big|u(\tilde{x}_l) - \tilde{u}_l^h\big|^2 \bigg)^{1/2}. 
\eeas
Note that $M = 1000$ is large enough so that further increasing its value does not change the RMS errors. 
Figure \ref{Fig-remark} (b) again shows that our method yields good accuracy for all points across the domain, not only at the grid points. 
\end{remark}}
\begin{example}\label{example5-2-2}
Consider the two-dimensional elliptic problem: 
\begin{equation}\label{elliptic-ex2}
\begin{aligned}
    (-\Dt)^{\fl{\ap}{2}} u(\bx) + u(\bx) = f(\bx),\qquad  &\mbox{for} \ \  \bx\in\Omega, \\
    u(\bx) = e^{-a^2|\bx|^2},\qquad  &\mbox{for} \ \ \bx\in\Omega^{c},
\end{aligned}
\end{equation}
for $a \in {\mathbb R}$. 
Choose the right-hand side function as $$f(\bx) = (2a)^\ap \Gamma\big(1+\fl{\ap}{2}\big){}_1F_1\Big(1+\fl{\ap}{2};\, 1; \, -a^2|\bx|^2\Big) + e^{-a^2|\bx|^2}, \qquad \mbox{for} \ \ \bx \in \Og, $$ such that the exact solution of (\ref{elliptic-ex2}) is  given by $u(\bx) = e^{-a^2|\bx|^2}$,  for $\bx \in \Og$.
\end{example}

Set the domain $\Og = (-1.5, 1.5)^2$. 
Table \ref{tab1:2d-sol-exp} presents the numerical errors of our method for different power $\ap$ and mesh size $h$, where we choose $a = 6$ in (\ref{elliptic-ex2}).  
It shows that numerical errors decrease quickly when reducing the mesh size $h$, and our method has a spectral accuracy. 
\begin{table}[ht!]
\centering
\begin{tabular}{|lllll|}\hline
             & $\ap = 0.5$ & $\ap = 1.0$ & $\ap = 1.7$  & $\ap = 2.0$ \\ \hline
$h = 1/8$    &3.512e-2       &0.0154       &1.227        &3.013\\
$h = 1/16$   &9.835e-8    &5.174e-7    &5.285e-6     &1.438e-5\\
$h = 1/32$   &3.100e-15   &1.420e-14   &3.365e-13    &1.288e-12 \\
\hline
\end{tabular}
    \caption{Numerical errors  $\|e_u^h\|_{l^2(\Og)}$ of our method in solving the 2D elliptic problem (\ref{elliptic-ex2}) on  $\Og = (-1.5, 1.5)^2$, where $a=6$ is chosen in (\ref{elliptic-ex2}).}
    \label{tab1:2d-sol-exp}
\end{table}
In this case, the solution $u \in C^\infty({\mathbb R}^2)$, and the observed spectral accuracy confirms our analytical results in Theorem \ref{thm3} (ii). 
We find that the maximum errors in this case occur at the center of the domain, i.e., around $\bx = {\bf 0}$. 

\begin{example}\label{example-coexist}
Consider the coexistence problem of normal and anomalous diffusion:
\bea\label{elliptic-ex3}
\big[\lambda_1(-\Dt)^{\fl{\ap_1}{2}} + (1-\lambda_1)(-\Dt)^{\fl{\ap_2}{2}} \big] u(\bx) = e^{-|\bx|^2}\cos^4\Big(\fl{3\pi}{2}|\bx|\Big)\chi_{\{|\bx|<1\}},
\eea
for $\bx \in \Og$, with homogeneous Dirichlet boundary conditions, i.e. $u(\bx) = 0$ for $\bx \in \Og^c$, where $\lambda_1 \ge 0$, and $\ap_1, \ap_2 \in (0, 2]$. 
Without loss of generality, we assume $\ap_1 \le  \ap_2$.  
The characteristic function $\chi_{\{|\bx|<1\}} = 1$ if $|\bx| < 1$, and otherwise it is 0 if $|\bx| \ge 1$.  
\end{example}

The nonlocal elliptic equation (\ref{elliptic-ex3}) can be viewed as  the steady state  of the coexistence of anomalous-anomalous (i.e., both $\ap_1, \ap_2 < 2$) or anomalous-normal (i.e., $\ap_1 < \ap_2 \le 2$) diffusion observed in many fields \cite{Javanainen2012, Lenzi2016, Zhang2012}. 
The exact solution of (\ref{elliptic-ex3}) is unknown. 
Hence, we use our numerical method to study the solution behaviors with $\Og = (-1, 1)^2$. 
Figure \ref{fig:ex6} presents the numerical solution for different $\lambda_1$, $\ap_1$ and $\ap_2$. 
\begin{figure}[ht!]
\centerline{
\includegraphics[width=3.0cm, height = 2.64cm]{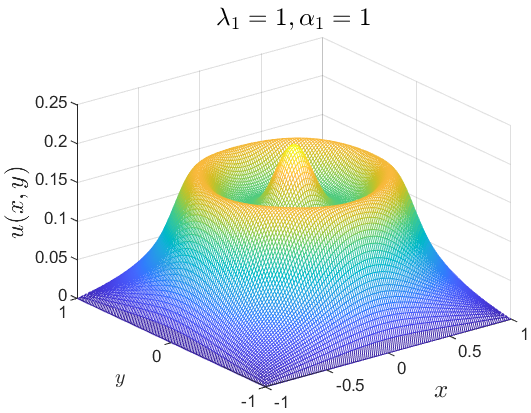}\hspace{1mm}
\includegraphics[width=3.0cm, height = 2.64cm]{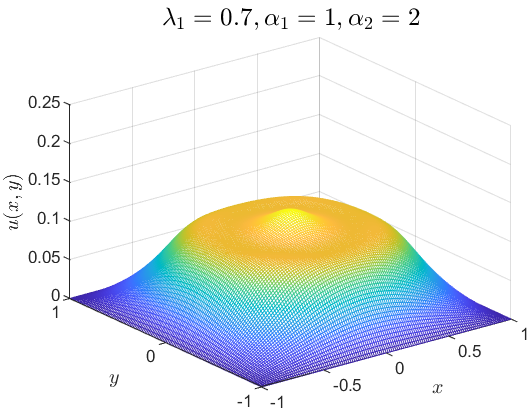}\hspace{1mm}
\includegraphics[width=3.0cm, height = 2.64cm]{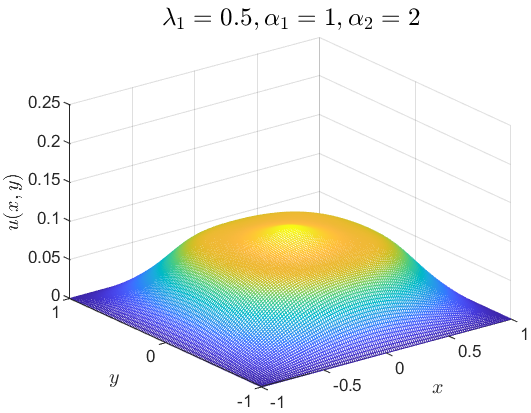}\hspace{1mm}
\includegraphics[width=3.0cm, height = 2.64cm]{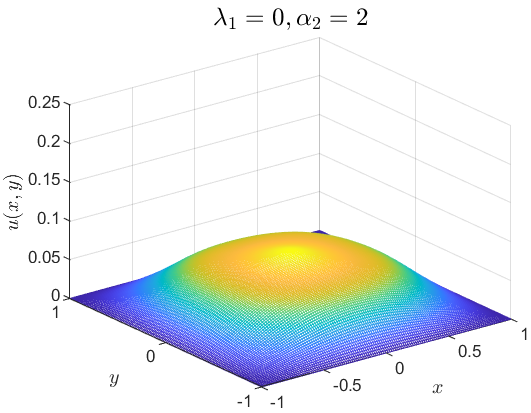}}
\vspace{3mm}
\centerline{
\includegraphics[width=3.0cm, height = 2.64cm]{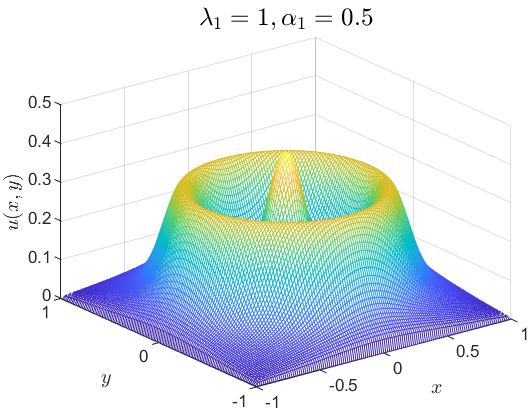}\hspace{1mm}
\includegraphics[width=3.0cm, height = 2.64cm]{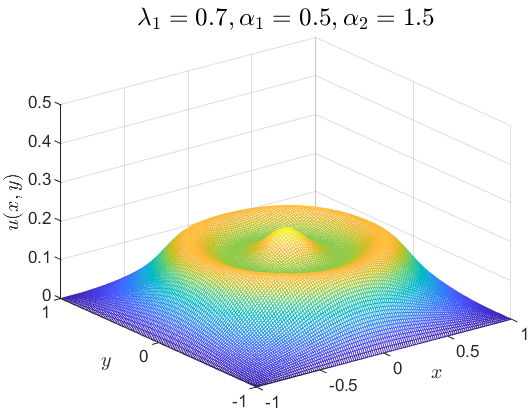}\hspace{1mm}
\includegraphics[width=3.0cm, height = 2.64cm]{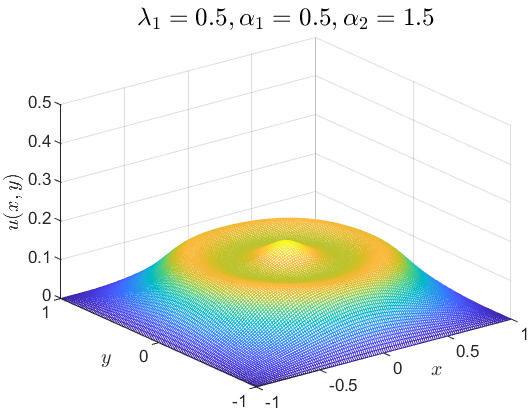}\hspace{1mm}
\includegraphics[width=3.0cm, height = 2.64cm]{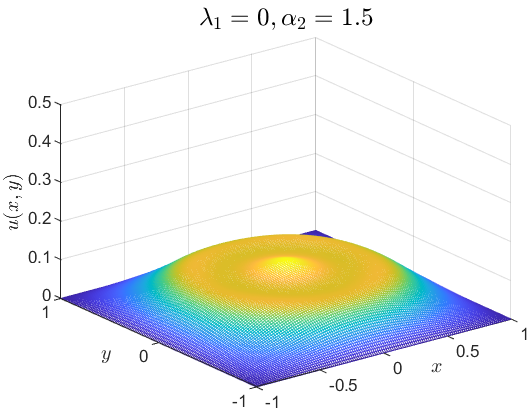}}
\caption{Solution of the coexistence problem (\ref{elliptic-ex3}) for different $\lambda_1, \ap_1$ and $\ap_2$. }
\label{fig:ex6}
\end{figure}
It shows that the solutions are radially symmetric, and its maximum value is reached at the center $\bx = {\bf 0}$. 

In pure classical diffusion cases (i.e., $\lambda_1 = 0$ and $\ap_2 = 2$),  the solution decreases gradually from the center to the boundary due to the homogeneous Dirichlet boundary conditions. 
In contrast, the solutions in the fractional cases are very different -- an annular layer is formed in between the center and boundary. 
The smaller the power $\ap$, the sharper the layer (cf. results of $\lambda_1 = 1$, $\ap_1 = 1$ and $\ap_1 = 0.5$).   
The solutions in mixed anomalous-normal or anomalous-anomalous cases are more complicated depending on the value of $\lambda_1$. 
\begin{figure}[ht!]
\centerline{
\includegraphics[width=5.68cm, height = 4.64cm]{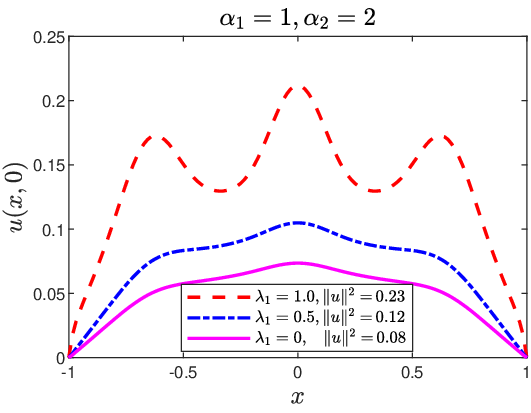}\hspace{4mm}
\includegraphics[width=5.68cm, height = 4.64cm]{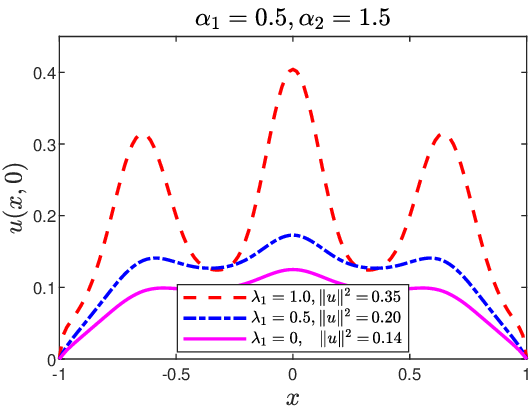}}
\caption{Comparison of solution at $y = 0$ for different $\lambda_1$,  $\ap_1$ and $\ap_2$,  where$\|u\|^2 = \int_{\Og} |u(\bx)|^2 \md\bx$. }\label{fig-ex6-2}
\end{figure}
In Figure \ref{fig-ex6-2}, we compare the solution at $y = 0$ for different $\lambda_1$ and $\ap_k$ (for $k = 1, 2$) and also the solution mass $\|u\|^2 = \int_{\Og}|u(\bx)|^2 \md\bx$.  
Generally, the smaller the value of $\lambda_1$, the smaller the magnitude of solution $u$, and the lower   the solution mass $\|u\|^2$.

\section{Conclusions}
\label{section6}

We proposed a novel and simple spectral method to discretize the $d$-dimensional (for $d \ge 1$) fractional Laplacian $(-\Dt)^\fl{\ap}{2}$. 
The key idea of our method is to apply the semi-discrete Fourier transforms to approximate the pseudo-differential definition of the fractional Laplacian.  
Our scheme can be viewed as a discrete pseudo-differential operator with symbol $|\bxi|^\ap$ and thus provides an exact discrete analogue of the fractional Laplacian $(-\Dt)^\fl{\ap}{2}$. 
Detailed numerical analysis was presented to study the  truncation errors under different conditions. 
It showed that our method can achieve a spectral accuracy if the function is smooth enough. 
Moreover,  the stability and convergence of our method in solving the fractional Poisson equations were analyzed.
In contrast to the existing Fourier pseudospectral methods in \cite{Duo2016, Kirkpatrick2016, Duo2021}, our method evades the constraint of periodic boundary conditions. 
Moreover, it results in a multilevel Toeplitz stiffness matrix,  and thus fast algorithms can be developed for efficient matrix-vector multiplication. 
The computational complexity is ${\mathcal O}(2N\log (2N))$, while the memory storage is ${\mathcal O}(N)$ with $N$ the total number of points.

Extensive numerical experiments were reported to demonstrate the effectiveness of our method. 
Our numerical study showed that the smoother the function to approximate, the higher the accuracy of our method.  
The observed accuracy rate verified our analytical results. 
The comparison of our method to finite difference methods further demonstrated its effectiveness. 
It showed that our method can be formulated similarly to finite difference methods. 
However, our method could achieve much higher accuracy, and the cost of computing matrix entries is independent of dimension $d$. 
The advantage of our method is more significant in high dimensions, as fewer points are demanded to obtain the desired accuracy. 
Consequently, our method considerably reduces the computational and storage costs. 
In addition,  the accuracy of our method in solving the fractional elliptic equations was numerically studied.  
We found that if the solution $u \in C^{\lfloor s\rfloor, \, s-\lfloor s\rfloor}(\bar{\Og})$,   our method has errors  $\|e_u^h\|_{l^\infty(\Og)} \sim {\mathcal O}(h^{s})$.  
We also applied our method to study the coexistence of anomalous-anomalous or anomalous-normal diffusion problems. 

\section*{Acknowledgements}
This work was partially supported by the US National Science Foundation under grant numbers DMS--1913293 and DMS--1953177. 

\bibliographystyle{plain}

\end{document}